 \documentclass[12pt,a4paper]{article}

\usepackage{amsmath}
\usepackage{mathrsfs}
\usepackage{amsfonts}
\usepackage{amssymb}
\usepackage{amsthm}
\usepackage{latexsym}
\usepackage[active]{srcltx}
\usepackage{a4wide}
\def\R{{\mathbb R}}
\def\C{{\mathbb C}}
\def\N{{\mathbb N}}

\def\G{{\mathbb G}}

\def\ZZ{{\mathcal  Z}}

\def\Id{{\mathbb I}}
\def\l{\langle}

\def\I{{\mathcal I}}
\def\B{{\mathcal B}}
\def\CC{{\mathcal C}}
\def\D{{\mathcal D}}

\def\H{{\mathcal H}}

\def\M{{\mathcal M}}

\def\S{{\mathcal S}}

\def\a{{\frak a}}

\def\g{{\frak g}}

\def\cc{{\frak c}}
\def\n{{\frak n}}
\def\p{{\frak p}}

\def\s{{\frak s}}

\def\al{\alpha}
\def\be{\beta}

\def\de{\delta}

\def\Om{\Omega}
\def\OM{\Omega}

\def\ga{\gamma}
\def\GA{\Gamma}

\def\ch{\chi}

\def\Si{\Sigma}

\def\va{\varphi}
\def\ps{\psi}

\def\ad{{\rm {ad}}}

\def\Ad{{\rm Ad}}

\def\ker{{\rm ker}}

\def\tr{{\rm tr}}

\def\Id{{\Bbb I}}

\def\inv{^{-1}}

\def\tr{{\text tr}}

\def\iy{\infty}

\def\ol#1{\overline{#1}}

\def\hb#1{\hbox{#1}}
\def\val#1{\vert #1\vert}

\def\no#1#2{\Vert #1\Vert_{#2} }

\def \nn{\nonumber}

\def\exp#1{\hb{exp}(#1)}

\def\ker#1{\hb{ker}(#1)}

\def\res#1{_{\vert #1}}
\def\inv{^{-1}}

\def\es{\emptyset}

\def\hb #1{\hbox{#1}}

\def\wh{\widehat}

\def\hb#1{\hbox{#1}}
\def\val#1{\vert #1\vert}

\def\ti{\times}
\def\ker#1{\hb{ker}(#1)}

\def\dim#1{{\rm dim}(#1)}

\def\l#1#2{L^{#1}{(#2)}}

\def\lef({\left(}
\def\rig){\right)}

\def\noi{\noindent}

\newtheorem{theorem}[subsection]{Theorem}

\newtheorem{remarks}[subsubsection]{Remarks}
\newtheorem{remark}[subsubsection]{Remark}
\newtheorem{definition}[subsubsection]{Definition}

\begin{document}

\title{A retract theorem for nilpotent Lie groups}

\author{Ying-Fen Lin, Jean Ludwig and Carine Molitor-Braun}


\date{}

\maketitle

\begin{abstract}
\footnote{keywords: nilpotent Lie group, irreducible
representation, co-adjoint orbit, Fourier inversion, retract,
compact group action -- 2010 Mathematics Subject Classification:
22E30, 22E27, 43A20} Let $G=\exp \g$ be a connected, simply connected nilpotent Lie
group. We show that for every $G $-invariant smooth sub-manifold $M $ of $\g^* $,  there exists an open relatively compact subset $\M $ of $M $ such that for any smooth adapted field of operators $(F(l))_{l\in M} $ supported in $G\cdot \M $ there exists a Schwartz function $f $ on $G $ such that $\pi_l(f)=op_{F(l)}$ for all $l\in M $.
This retract theorem can then be used to show that for every Lie group $\bf G $ of automorphisms of $G $ containing the inner automorphisms of $G $ with locally closed $\bf G$-orbits in $\g^* $,  the proper $\bf G$-prime two-sided closed ideals of $L^1(G)$ are the kernels of $\bf G $-orbits in $\widehat G $.
\end{abstract}

\section{Introduction}
For a connected, simply connected, nilpotent Lie group
$G$, the description of its  spectrum and of the  Fourier inversion theorem is due to Kirillov \cite{Kir}, who showed that the dual space $\wh G $ of $G $ is in one-to-one correspondence with the space $\g^*/G $ of co-adjoint orbits of $G $. R. Howe proved in \cite {Ho} that for every irreducible unitary representation $(\pi,\H_\pi) $ of $G $ and every smooth linear operator $a $ on $\H_\pi $ there exists a Schwartz function $f_a $ on $G $ such that $\pi(f_a)= a$.  He also showed that the mapping $a \mapsto f_a $ is linear and continuous with respect to the Fr\'echet topology of the space $\B^\iy(\H_{\pi}) $ of smooth linear operators on $\H_\pi $ and the  Fr\'echet space $\S(G) $ of Schwartz functions on $G $.

In this paper, we study a version of the Fourier inversion theorem for nilpotent Lie groups. More precisely, we generalise the result of R. Howe's  mentioned above by constructing a continuous retract from the space of adapted smooth kernel functions defined on  a smooth $G $-invariant sub-manifold $M$ of $\g^* $ and supported in a subset $G\cdot \M $ of $M $, where  $\M $ is a relatively compact open subset of $M $, into the space $\S(G) $.
We will prove this result, which we call the retract theorem, by proceeding an induction on the length $\val I $ of the largest index set $I $ for which $(\B\ti \g^*)_I\cap M\ne \es $. In order to do so we will apply the variable group techniques developed in \cite{Lu-Z}, which have already been used in \cite{Lu.Mu.}.

Once we have the retract theorem, we can apply it to study the $\bf G $-prime ideals of the Banach algebra $\l1G $.
Here $\bf G$ denotes a Lie subgroup of the automorphism group of $G $ with the property that the $\bf G $-orbits in $\g^* $ are all locally closed.
The retract theorem implies that the Schwartz functions contained  in the kernel of a $\bf G$-orbit $\Om $ in $\wh G $ are dense in the $\l1G $-kernel of $\Om $.
Using the methods in \cite{Lu.Mo.1}, it follows that every $\bf G$-prime ideal in $\l1G $ is the kernel of such an $\bf G$-orbit $\Om $.
One can for instance use this result for the study of bounded irreducible representations $(\pi,X) $ of a Lie group $\bf G $ on a Banach space $X $. Restricting  the representation $\pi $ to the nilradical $G $, one obtains the kernel $\ker {\pi\res G} $ of $\pi\res G $ in  the algebra $L^1(G) $. This ideal  $\ker {\pi\res G} $ is then $\bf G $-prime. If $\ker {\pi\res G} $ is  given as the kernel in $L^1(G) $ of a $\bf G $-orbit ${\bf G}\cdot \pi_0 \subset \wh G$ for some $\pi_0 \in \wh G$, then one use $\pi_0 $ to   make an  analysis of $\pi $ as Mackey did in the case of  unitary representations.

For connected, simply connected, nilpotent Lie
groups, J. Ludwig showed in \cite {Lu.} that the closed
prime ideals  of $L^1( G )$ coincide with the
kernels  of the irreducible unitary representations. In  $1984$, D. Poguntke studied the action
of an abelian compact group $K$ on a nilpotent Lie group \cite{P1} and characterised the $K $-prime ideals as kernels of $K $-orbits.
In \cite{La.Mo.1}, R. Lahiani and C. Molitor-Braun identified the $K$-prime ideals with hull contained in the generic part of the dual space of $G $ for a general compact Lie subgroup $K $ of the automorphism group of $G$. In \cite {Lu.Mo.1} and \cite{Lu-Mo.2}, it was shown that
for an exponential Lie group $\bf G$, the $\bf G$-prime ideals are also kernels of $\bf G$-orbits. In this way the bounded irreducible Banach space representations of an exponential Lie group could be determined.

The paper is organised in the following way: in Section \ref{intro} we recall the definition of induced representations and of kernel functions, we explain the notion of variable nilpotent Lie groups and their Lie algebras, of index sets for co-adjoint orbits and  of adapted kernel functions on a $G$-invariant sub-manifold of $\g^* $. In Section \ref{retrtheo}, we introduce the main theorem of the paper, the Retract Theorem, and in Section  \ref{proof} we present the proof of the theorem, dividing it  into several steps. As an application of the Retract Theorem, in the last section (Section \ref{primeideals}) we show that every $\bf G$-prime ideal in $\l1G $ is the kernel of a $\bf G$-orbit.

\section{Notations and generalities}\label{intro}

\subsection{Representations and kernel functions}
Let $G=\exp \g$ be a connected, simply connected, nilpotent Lie
group and $\g$ be its Lie algebra. All the irreducible unitary representations of $G$ (and
hence of $L^1(G)$) are obtained (up to equivalence) in the
following way: Let $l\in \g^*$ and  $\p=\p(l)$ be an arbitrary
polarisation of $l$ in $\g$ (a maximal isotropic subalgebra of
$\g$ for the bilinear form $(X,Y)\mapsto \langle l,[X,Y]\rangle$). Let
$P(l)=\exp{\p(l)}$. The induced representation denoted by $\pi_l:=
\text{ind}^G_{P(l)}\chi_l$ on the Hilbert space $\frak H_l$,
\begin{eqnarray*}
 \frak H_l=L^2(G/P(l), \chi_l)&:=&\{\xi:G\to\C; \, \xi \text{ measurable },\xi(gp)=\ch_l(p\inv)\xi(g), g\in G, p\in P(l),\\
\nn  &  &\no{\xi}2^2=
\int_{G/P(l)}\val{\xi(g)}^2d\dot g <\iy\},
 \end{eqnarray*}
where $d\dot g $ is the invariant measure on $G/P(l)$, is unitary and irreducible. Here
$\chi_l$ is the character defined on $P(l)$ by $\chi_l(g) =
e^{-i \langle l,\log g \rangle}$ for all $g\in P(l)$. Two different polarisations
for the same $l$ give equivalent representations. The same is true for the case of 
two linear forms $l$ and $l'$ belonging to the same co-adjoint orbit.

One particular way to obtain a polarisation is the
following: Let $\{Z_1, \dots, Z_n\}$ denote a Jordan-H\" older
basis of $\g$, for $1\leq k\leq n$, let $\g_k:=\text{span}\{Z_k, \dots, Z_n\}$ be the linear
span of $Z_k, \dots, Z_n$ and $l_k = l\vert_{\g_k}$ for all $l\in
\g^*$. The polarisation $\p(l)_\ZZ=\p(l):=\sum_{k=1}^n \g_k(l_k)$ of $l$ in $\g$, with $\g_k(l_k):=
\{U\in \g_k; \langle l,[U,\g_k] \rangle \equiv 0\}$, is called the Vergne polarisation at $l$ with respect to
the basis $Z_1, \dots, Z_n$. We refer to \cite{Co.Gr.} for more details on the theory of
irreducible representations of nilpotent Lie groups.

Let $\pi_l=\text{ind}_{P(l)}^G \chi_l$. The corresponding
representation of $L^1(G)$, also denoted by $\pi_l$, is obtained
via the formula $\pi_l(f)\xi:= \int_G f(x) \big (\pi_l(x)\xi
\big)dx$, for all $\xi \in \frak H_l$. If $f\in {L^1(G)}$,
then $\pi_l(f)$ is a kernel operator, i.e. it is of the form
$$\big (\pi_l(f)\xi \big )(g) = \int_{G/P(l)} F(l,g,u) \xi(u)du,$$
where $F$ is the operator kernel given by
$$F(l,g,u)=\int_{P(l)}f(ghu^{-1})\chi_l(h)dh \quad  \text{for } \, g,u\in G.$$
If $f$ is a Schwartz function, then the kernel function $F$ belongs to
${\mathcal C}^{\infty}$ and satisfies the covariance relation
$$F(l, gh, g'h')= \overline{\chi_l(h)} \chi_l(h') F(l, g, g') \quad \text{for } \, h,h' \in P(l) \text{ and } g,g' \in G,$$
and is a Schwartz function on $G/P(l) \times G/P(l)Ã$.

\subsection{Group actions}\label{compact}
Let $G=\exp{\g}$ be a connected, simply connected, nilpotent Lie
group and $A$ be a Lie subgroup of the automorphism group $\text{Aut}(G)$ of $G$ acting smoothly on $G $. This action will be denoted by
$$\begin{array}{ccc}
\nn A \times G &\mapsto &G\\
\nn (a,x) &\mapsto &a\cdot x.
\end{array}$$ The action of $A$ on $G$ induces naturally actions of $A$ on $\g$, $\g^*$, $\wh G$, $L^1(G)$, and on ${\mathcal S}(G)$.
These group actions will lead to examples for our
retract theory and provide an important application of retracts.

\subsection{Variable Lie algebras and groups}\label{variable}

We will prove our main theorem by induction; in our proofs, new parameters and new variations will appear. This may be handled most easily by the concept of \emph{variable Lie structures}. Such structures were already
considered in \cite{Le.Lu.}, \cite{Lu-Z}, \cite{Lu.Mu.} and \cite{Lu.Mo.Sc.}, among others.

\begin{definition}
\begin{enumerate}
\item  Let $\g$ be a real vector space of finite dimension $n$ and
${\mathcal B}$ be an arbitrary nonempty set. We say that $( {\mathcal
B},\g)$ is a {\it variable (nilpotent) Lie algebra} if
\begin{enumerate}
\item For every $\beta \in {\mathcal B}$, there exists a Lie bracket
$[ \cdot, \cdot]_{\beta}$ defined on $\g$ such that
$\g_{\beta}:=(\g, [ \cdot, \cdot]_{\beta})$ is a nilpotent Lie
algebra.

\item There exists a fixed basis $\ZZ=\ZZ^0=\{ Z_1=Z_1^0, \dots, Z_n=Z_n^0\}$ of $\g$
such that the structure constants $a^k_{ij}(\beta)$ defined by
$$[Z_i, Z_j]_{\beta} := \sum_{k=1}^n a^k_{ij}(\beta) Z_k$$ satisfy
the following property: For all $\beta \in {\mathcal B}$ and
$k\leq \max\{i, j\}$, $a^k_{ij}(\beta)=0$. This means that $\{Z_1,
\dots, Z_n\}$ is a Jordan-H\" older basis for $\g_{\beta}=(\g, [\cdot, \cdot]_{\beta})$.
\end{enumerate}

\item Assume that ${\mathcal B}$ is a  smooth
manifold. If the structure constants $a^k_{ij}(\beta)$ vary
 smoothly on $ \B $, we say that $({\mathcal B},
\g)$ is a smooth variable (nilpotent) Lie algebra.
\end{enumerate}

We will denote $({\mathcal B}, \g) = (\g, [\cdot,
\cdot]_{\beta})_{\beta \in {\mathcal B}}$ for the variable Lie algebra.
\end{definition}

\rm For the rest of the paper we will assume that all
variable Lie algebras are smooth. If ${\mathcal B}$ is reduced to a singleton, we have in
fact no dependency on $\beta$ in ${\mathcal B}$ but a fixed Lie algebra.
To each variable Lie algebra, we associate a variable Lie group $G_{\beta}$. The variable Lie group
$\G:=(G_{\beta})_{\beta}$ may be identified with the
collection of Lie algebras $(\g, [\cdot, \cdot]_{\beta})_{\beta}$
equipped with the corresponding Campbell-Baker-Hausdorff multiplications.
If $\G=(G_{\beta})_{\beta}$ is a (smooth) variable
Lie group endowed with a fixed Jordan-H\" older basis, then the
corresponding Vergne polarisations, induced representations and
operator kernels all depend on $\beta \in {\mathcal B}$ and $l\in \g^*$.

\subsection{Ludwig-Zahir indices}\label{step}
Let $({\mathcal B}, \g)$ be a smooth variable Lie algebra. We
assume that $\g$ is equipped with a fixed basis
$\ZZ=\ZZ^0=\{Z_1=Z_1^0, \dots, Z_n=Z_n^0\}$, which is a Jordan-H\"
older basis for every $(\g, [\cdot, \cdot]_{\beta})$.

Let $(\beta, l)\in {\mathcal B} \times \g^*$.
The Ludwig-Zahir indices $ I(\be, l)$ defined in \cite{Lu-Z} can be obtained in the following way: Let $\g_{\beta}(l):= \{ U\in
\g; \langle l,[U,\g]_{\beta} \rangle \equiv 0\}$ be the stabiliser of $l$ in
$\g_{\beta}=(\g, [\cdot, \cdot]_{\beta})$ and let $\a_{\beta}(l)$ be
the maximal ideal contained in $\g_{\beta}(l)$. If
$\a_{\beta}(l)=\g_{\beta}(l)=\g$, then $\chi_{(\beta,l)}(x):=
e^{-i \langle l, \log_{\beta}x \rangle}$ is a character on $G_{\beta}$ and nothing has to be done. In this case, there are no Ludwig-Zahir indices, i.e. $ I(\be, l)=\es $. Otherwise, let
\begin{eqnarray}
\nn j_1(\beta,l) &=& \max \{j\in \{1, \dots ,n\};~Z_j^0 \not\in
\a_{\beta}(l)\}, \, \text{ and}\\
\nn k_1(\beta,l) &=& \max \{ k\in \{1, \dots, n\};~\langle l,
[Z^0_{j_1(\beta,l)}, Z^0_k]_{\beta}\rangle \neq 0\}.
\end{eqnarray}
We let
\begin{eqnarray*}
X_1(\be,l)&:&=Z^0_{k_1(\beta,l)}, \\
Y_1(\be, l)&:&=Z^0_{j_1(\beta,l)}, \\
Z_1(\be, l)&:&= [Z^0_{k_1(\beta, l)},Z^0_{j_1(\beta,l)}] _\be, \, \text{ and} \\
c(\be,l)&:&=\langle{l},{Z_1(\be,l)}\rangle.
\end{eqnarray*}
We then consider
\begin{eqnarray}\label{gonbel}
\g_1(\beta,l):= \{U \in \g; ~\langle l,[U,Y_1(\beta,l)]_\be\rangle =0\}
\end{eqnarray}
which is an ideal of co-dimension one in $\g_{\beta}$.

A Jordan-H\" older basis  of $(\g_1(\beta,l), [\cdot,
\cdot]_{\beta})$ is given by $\ZZ^1(\be,l)=\{ Z^1_i(\beta,l);~i
\neq k_1(\beta,l)\}$ with
\begin{eqnarray}\label{defzonzbek}
Z^1_i(\beta,l) := Z_i^0 -\frac{ {\langle l,[Z^0_i,
Y_1(\beta,l)]_{\beta}\rangle}}{c(\be,l)} X_1{(\beta,l)},  \quad i\neq k_1(\beta,l).
\end{eqnarray}
One sees that $Z^1_i(\beta,l) = Z^0_i$, if $i>k_1(\beta,l)$.
As previously we may now compute the indices $j_2(\beta,l),
k_2(\beta,l)$ of $l_1:=l\vert_{\g_1(\beta,l)}$ with respect to
this new basis and construct the corresponding subalgebra
$\g_2(\beta,l)$ with its associated basis
$\{Z^2_i(\beta,l)~; ~i \neq k_1(\beta,l), k_2(\beta,l)\}$. This
procedure stops after a finite number $r$ of steps. Let
$$I_\ZZ(\beta,l) = I(\be, l)=\big ((j_1(\beta,l), k_1(\beta,l)), \dots,
(j_r(\beta,l), k_r(\beta,l)) \big ),$$ which is called the
\emph{Ludwig-Zahir index} of $l$ in $\g_{\beta}$ with respect to the
basis $\{Z_1, \dots, Z_n\}$. The construction in \cite{Lu-Z}
shows that the final subalgebra $\g_r(\beta,l)$ obtained by this
construction coincides with the Vergne polarisation of $l$ in
$\g_{\beta}$ with respect to the basis $\ZZ^0$ (see also
\cite{Lu.Mu.}, \cite{Lu.Mo.Sc.}).
Note that the length $ \val I=2r $ of the index set $I= I(\be, l) $ gives us the dimension of the co-adjoint orbit $ \Ad^*(G_\be)\ell $.
The vectors $Y_1(\be,l),\cdots, Y_r(\be,l) $ together with the stabiliser $\g_\be(l) $ of $l $ in $\g_\be $ span the polarisation $\p_\be(l)=\g_r(\be,l) $ and
\begin{eqnarray*}
 \g=\oplus_{i=1}^r \R X_i(\be,l)\oplus_{i=1}^r \R Y_i(\be,l)\oplus \g_\be(l).
 \end{eqnarray*}

Let us introduce the following notations: For any
index set $ I\in (\N^2)^r \equiv \N^{2r}$ with $r=0,\cdots, \dim {\g}/2$, we let
\begin{eqnarray*}
&(\B\ti\g^*)_I:=\{(\be,l)\in\B\ti\g^*; I(\be,l)=I\} \quad \textrm{ and }\\
&(\B \ti \g^*)_I \cap \Sigma_I := \{(\beta,l) \in (\B\ti \g^*)_I; \,
l(Z_{j_i})=l(Z_{k_i})=0 \textrm{ for } 1 \leq i \leq r\}.
\end{eqnarray*}
This last line corresponds to the Pukanszky section associated to
the index $I$. In fact, in \cite{Lu.Mo.Sc.} it was proved that
the indices $j_s(\beta,l), k_s(\beta,l)$ coincide with the
Pukanszky indices of the given layer (if one does not make any
distinction between the $j$'s and the $k$'s). For many $I$'s, the
subset $(\B\ti\g^*)_I$ is empty. Hence it is reasonable to define
\begin{eqnarray*}
\I:=\Big\{I\in \bigcup_{j=0}^{\dim {\g}/2} (\N^2)^{j}; \,
(\B\ti\g^*)_I\ne\es\Big\}, \quad \textrm{and} \quad \B\ti\g^*=
\dot\bigcup_{I\in \I}(\B\ti\g^*)_I.
\end{eqnarray*}
This gives a partition of $ \B\ti\g^*$ into the different layers
$(\B\ti\g^*)_I $. The set $\I$ may  be ordered lexicographically:
if $I=\{(j_1,k_1),\cdots, (j_r,k_r)\},I'=\{(j_1',k_1'),\cdots,
(j_{r'}',k_{r'}')\}\in\I $, we say that $I<I' $ if either  $2r=\val
I< \val{I'}= 2r' $  or there exists $a\in \{ 1, \dots, r\}$ such that
$$( j_s, k_s) = (j_s', k_s') \textrm{ if } s<a \textrm{ and } (
j_a, k_a) < (j_a', k_a'),$$ which means that
$$\textrm{either }  j_a < j_a'  \textrm{ or } (j_a = j_a' \textrm{ and }  k_a <
k_a').$$ This allows us to define
$$(\B \ti \g^*)_{\leq I} := \{(\beta,l) \in (\B \ti \g^*)_J; \,J \leq I\}= \bigcup_{J\leq I}(\B\ti\g^*)_J.$$
By induction on the length of the index sets, it is easy to see that for every $I\in \I $ there exists  a smooth function $P_I $ on $\B\ti\g^* $, which is polynomial in $l $ for fixed $\be\in\B $ such that
\begin{eqnarray}\label{BgstarI}
 (\B\ti \g^*)_I=\{(\be,l); P_{I'}(\be,l)=0 \text{ for }I'>I \text{ and } P_I(\be,l)\ne 0\}.
\end{eqnarray}

\subsection{Co-adjoint orbits}\label{coadjoint}

For any index set $I $, we consider the subspace $\s_I $ of $\g^* $ which is given by
\begin{eqnarray*}
 \s_I=\text{span} \{Z_j^*;~ j\in I\}.
 \end{eqnarray*}
For each $\be\in\B $, let
\begin{eqnarray*}
 \Si_{\be,I}:=\{(\be,l)\in (\{\be\}\ti\g^*)_I;~ l\in\s_I \}.
 \end{eqnarray*}
Then $\Si_{\be,I} $ is locally closed in $\s_I $, since we have the  smooth functions $P_{I'},I'\in\I$, defined on $\B\ti\g^*$ as in \eqref{BgstarI}.

Let $d:=\val I $. For $l\in\g^* $, let

\begin{eqnarray*}
 \Om_{\be,l}=\{Ad_\be^*(g)l; \, g\in G\}
 \end{eqnarray*}
be the $G_\be $-orbit of $l $.
Then
 \begin{eqnarray*}
 \dim {\OM_{\be,l}}=d \quad \text{for} \quad l\in(\B\ti\g^*)_I.
 \end{eqnarray*}
There exist functions $p_j:(\B\ti\g^*)_I\ti \R^d\to\R, j=1,\cdots, n, $ which are rational in $l\in \g^* $ and polynomial in $z\in \R^d $ for fixed $\be\in \B$ such that for every
$(\be, l)\in (\B\ti\g^*)_I $,
\begin{eqnarray*}
 \OM_{\be,l}=\Big\{\sum_{i=1}^n p_{i}(\be,l,z) Z_i^*; \, z\in \R^d\Big\}.
 \end{eqnarray*}
Furthermore if we write $I=\{i_1<\cdots<i_d\} $, then
\begin{eqnarray*}
 p_{i_j}(\be,l,z)= z_j \quad \text{for} \, \ j=1,\cdots,d,
 \end{eqnarray*}
and for $i\not\in I$, we have
\begin{eqnarray*}
p_{i}(\be,l,z)= \langle l ,Z_i\rangle +p'_i(\be,l,z_1,\cdots, z_j), \quad i_j<i<i_{j+1}.
 \end{eqnarray*}

\begin{definition}\label{}
\rm   A subset $M $ of $\B\ti\g^* $ is called \textit{$G $-invariant} if for every $(\be,l)\in M $ the element $g\cdot (\be,l):=(\be, \Ad^*_\be(g)l) $ is also contained in $M. $
\end{definition}

\subsection{Schwartz functions}\label{function}
Let $r \in \N \setminus \{0\}$, we define the space of (generalised) Schwartz functions ${\mathcal
S}( \R^r, {\mathcal B}, G) \equiv {\mathcal S}( \R^r, {\mathcal
B}, \g) \equiv {\mathcal S}( \R^r, {\mathcal B}, \R^n)$ to be the
set of all functions $f$ from $\R^r \times {\mathcal B} \times G$
to $\C$ such that the function $\tilde f$ defined by $$\tilde
f(\alpha,\beta,(x_1, \dots, x_n)):= f(\alpha,\beta,
\hbox{exp}_{\beta}(x_1Z_1+ \cdots + x_nZ_n)) \quad \text{for } \, \al\in \R^r, \be\in \B$$
is smooth on $\R^r \times {\B} \times \R^n$ and that
\begin{eqnarray}
\nn \Vert \tilde f\Vert_{K, T_1, \dots, T_s,
A_1, A_2, B_1, B_2} &=& \sup_{\beta \in K; \alpha \in \R^r; x\in
\R^n} \Big [\sup_{\vert r_i\vert \leq A_i; \vert s_j\vert \leq
B_j; i,j\in \{1,2\}} \vert \alpha^{r_1} x^{s_1} \\
\nn &{ }& \quad \quad \quad \quad T_1 T_2 \cdots T_s
\frac{\partial^{r_2}}{\partial \alpha^{r_2}}
\frac{\partial^{s_2}}{\partial x^{s_2}} \tilde f(\alpha,\beta,
(x_1, \dots, x_n)) \vert \Big ]\\
\nn &<& + \infty,
\end{eqnarray}
for any compact subset $K$ of ${\B}$, any finite collection
$T_1, \dots T_s$ of smooth vector fields defined on the manifold
${\B}$, and any $A_1, A_2, B_1, B_2 \in \N$. The function space
${\mathcal S}(\R^r, {\mathcal B}, G)$ is equipped with the
topology defined by the collection of all these semi-norms. One
may of course also use coordinates of the second kind to define
the semi-norms on ${\mathcal S}(\R^r, {\mathcal B}, G)$. Note that the space
${\mathcal S}(\R^r, {\mathcal B}, G)$ does not depend on the choice of the Jordan-H\" older basis.

\subsection{Kernel functions}\label{kerfunction}

Let $S$ be a subset of $ \B\ti\g^*  $ and $L$ be  a smooth manifold. We say that a mapping $F:S\to L $ is smooth, if the restriction of $F $ to any smooth manifold  $N $ contained in $S $ is smooth.

Let $\B\ti\g^* $ be a smooth variable nilpotent Lie group with Jordan-H\"older basis $\ZZ $.
For any $(\be,l)\in \B\ti \g^* $ denote the Vergne polarisation at $(\be,l) $ associated to $\ZZ $.
We put $\pi(\beta,l):=\textrm{ind}^G_{P(\beta,l)}
\chi_l$, with $P(\beta,l):=\textrm{exp}_{\beta} \p(\beta,l)$, for
the corresponding family of induced unitary representations. Then the mapping $(\be,l)\mapsto \p(\be,l) $ is smooth on each subset $(\B\ti\g^*)_I $.
For each index set $I $ with length $d_I$ and $(\be,l)\in \B\ti\g^* $, choose a Malvev basis $R(\be,l)=\{R_1(\be,l),\cdots, R_{d_I}(\be,l)  \} $ of $\g $ relative to $\p(\be,l) $, such that the mappings $(\be,l)\mapsto R(\be,l) $ are smooth on the different layers $(\B\ti\g^* )_I$.

\begin{definition}\label{kernelf}
Let $M$ be any smooth $G $-invariant manifold of $\B \ti \g^*$ and let $r\in\N $.   We denote by
${\mathcal D}^c_{M,r}$ the space of all  functions $F: \R^r \times { M} \times G \times G
\to \C$ satisfying the following conditions.

\begin{enumerate}\label{}
\item  $F $ satisfies the covariance condition for every $(\be,l)\in M $ with respect to  $\p(\be,l) $, i.e.
\begin{eqnarray*}
 F(\alpha, (\beta,l), x \cdot_{\beta}p, y\cdot_{\beta}q)
= \overline {\chi_l(p)} \chi_l(q) F(\alpha,(\beta,l), x,y))
 \end{eqnarray*}
 for
all $\alpha\in \R^r$, $p,q \in P(\beta,l)$ and $x,y\in G $.
\item  The function $F $ satisfies  the following compatibility
condition
\begin{eqnarray*}
 F(\alpha, (\beta, \text{Ad}^*_{\beta}(g)l), x,y)=F(\alpha,
(\beta,l), x\cdot_{\beta} g, y \cdot_{\beta} g),
 \end{eqnarray*}
for $\al\in\R^r, (\be,l)\in M$ and $x,y,g\in G $.
 This compatibility condition reflects
the unitary equivalence of the representations $\pi_{(\beta,l)}$
and $\pi_{(\beta, \text{Ad}^*_{\beta}(g)l)}$.
\item The support of $F$ in $(\beta,l)$ is compact modulo $G $, i.e. there exists a compact subset $C $ of $M $ such that $F(\cdot, (\beta,l), \cdot, \cdot)$ is $0$ outside the  subset of $G\cdot {C}$.
\item The function $F$ has the  Schwartz space property,
 i.e. for any $I\in \I $ the function $F\res{\R^r\ti M\cap (\B\ti\s_I)\ti G\times G }$ is smooth and that
\begin{eqnarray}
\nn \Vert F \Vert_{ D, A_1, A_2, B_1, B_2, C_1, C_2} &:=&
\sup_{(\beta,l) \in M, \alpha \in \R^r, x,x'\in \R^r} \Big [
\sup_{ \vert r_i\vert \leq A_i, \vert s_j\vert \leq B_j, \vert
t_k\vert \leq C_k; i, j, k \in \{1,2\}} \vert
\alpha^{r_1} x^{s_1} (x')^{t_1} \\
\nn &{ }&      D_{(\beta,l)}
\frac{\partial^{r_2}}{\partial \alpha^{r_2}} \frac
{\partial^{s_2}}{\partial x^{s_2}} \frac{\partial^{t_2}}{\partial
(x')^{t_2}}  \tilde F(\alpha, (\beta,l), x, x')\vert\Big ]< \iy, \\
\end{eqnarray}
where $$\tilde F(\alpha, (\beta, l), x, x'):= F(\alpha, (\beta,l),
\hbox{exp}_{\beta}(x_1R_1) \cdots \hbox{exp}_{\beta}(x_rR_r),
\hbox{exp}_{\beta}(x'_1R_1) \cdots \hbox{exp}_{\beta}(x'_rR_r)),$$
 for any smooth differential operator
$D=D_{(\beta,l)}$ on the manifold $M$, and any $A_1, A_2$, $B_1,
B_2$, $C_1, C_2 \in \N$.
 \end{enumerate}
\end{definition}

\rm
\noi The
space ${\mathcal D}^c_{M,r}$ will be equipped with the
topology defined by the collection of all these semi-norms. This
does of course not depend on the choice of the smooth Malcev basis
of $\g$ with respect to the smooth family of Vergne polarisations.

\begin{definition}\label{adapted}
Let $M\subset \B\ti\g^* $. A field $F=(F(\be,l))_{{(\be,l)}\in M} $ of kernel functions is called \textit{adapted} if it satisfies the conditions in Definition \ref{kernelf}.
 \end{definition}

For an adapted field of kernel functions $F $ on $M $, denote by
$op_F $ the field of smooth operators defined through their kernel functions. For $(\be,l)\in M $, the operator $op_{F(\be,l)} $ acts on the space $L^2(G/P(\be,l),\chi_{(\be,l)}) $ in the following way:
\begin{eqnarray*}
 op_{F(\be,l)}\xi(g)=\int_{G/P(\be,l)}F(\be,l)(g,x)\xi(x)d\dot x.
 \end{eqnarray*}

\begin{remarks}{\rm
a) If we impose the condition that the support of $(\beta,l)$ be contained in the set $G\cdot C_0 $ for a
fixed  subset $C_0$ of ${M}$, we will denote the
space of kernel functions by ${\mathcal D}^{C_0}_{M}$.

\noi b) One has a similar definition of the kernel functions if
one takes another smooth family of polarisations together with a
smooth family of Malcev bases.}
\end{remarks}

\section{The retract theorem}\label{retrtheo}

In this section, we state the main theorem of the paper which will be proved in the next section.

\begin{theorem}\label{retract} Let $\B\ti G $ be a smooth variable nilpotent Lie group, $I=\{(j_1,k_1)< \cdots< (j_r, k_r)\} $ be an index set  and let  ${ M}$ be a smooth $G $-invariant
sub-manifold of $\B\ti \g^*$ contained in $(\B\ti \g^*)_{\leq I}$ such that $M_I:=M\cap (\B\ti \g^*)_{I}\ne\es $.  Let $\pi(\beta,l)$ be defined as previously from the
smooth family of Vergne polarisations for $(\beta,l) \in M$. Then there exists an open nonempty relatively compact subset ${\mathcal M}\subset M_I$ with closure $\ol{\mathcal M}$ contained in $M_I$ such that
the following holds: For any adapted kernel function $F \in {\mathcal
D}^\M_{M}$, there is a function $f$
in the Schwartz space ${\mathcal S}(\R^r,{\mathcal B}, G)$ such
that $\pi_{(\beta,l)}(f(\alpha,\beta,\cdot))$ has $F(\alpha,
(\beta,l), \cdot, \cdot)$ as an operator kernel for all $(\alpha,
(\beta,l))\in \R^r \ti { M}$. Moreover the mapping $F\mapsto f$ is continuous with respect to the
corresponding function space topologies.
\end{theorem}

\noi If the variation is trivial, we get the following theorem.

\begin{theorem}\label{retract1} Let $ \g $ be  a nilpotent Lie algebra with Jordan-H\"older basis $\ZZ $.
  Let ${ M}$ be a smooth $G $-invariant sub-manifold of $\g^*$. Let $I:=\max\{J\in \I_\ZZ:\ M\cap \g^*_J\ne \es\} $.  Let $\pi_l=\pi(l)$ be defined as
previously from the smooth family of Vergne polarisations for $l\in M$. Then there exists an open, relatively compact nonempty subset $ {\mathcal M}\subset \g^*_{I} $ of $ M $ such that ${\mathcal M} \subset \overline
{\mathcal M}\subset M_I$, $\overline {\mathcal M}$ is compact and that the
following holds: For any kernel function $F \in {\mathcal D}^\M_{M}$,
there is a function $f$ in the Schwartz space ${\mathcal S}(G)$ such that $\pi_{l}(f)$ has $F( l, \cdot,
\cdot)$ as an operator kernel for all $l \in {M}$. Moreover, the Schwartz function $f$ may be
constructed such that the mapping $F\mapsto f$ is continuous with respect to the corresponding
function space topologies.
\end{theorem}

\begin{remark}\label{genericsi}
\rm   If $M $ is contained in $\g^*_{I_{max}} $, where $I_{max}$ is the maximal index set in $\I $, then we have the following (well known) result.
\end{remark}

\begin{theorem}\label{generic}
Let $\B\ti G$ be  a simply connected, connected smooth variable nilpotent Lie group 
and $M=(\B\ti \g^*)_{gen}:=(\B\ti \g^*)_{I_{max}} $ be the
space of generic co-adjoint orbits. Let $\M $ be an open relatively compact subset of $M $ such that $\ol \M\subset M.  $ For every adapted field of kernel functions $F\in \D^\M_M $, there exists a unique  Schwartz function $f=R(F):G\to \C $ such that
\begin{eqnarray*}
 \pi_{(\be,l)}(f)=op_{F(\be,l)} \, \text{ for any } \, (\be,l)\in \B\ti\g^*,
 \end{eqnarray*}
and the mapping $F\mapsto R(F) $ is continuous.
\end{theorem}

\begin{proof} It suffices to apply the Fourier inversion formula.
For each $F\in \D^\M_M $, let
\begin{eqnarray*}
 f(\be,g)=R(F)(\be,g):=\int_{\Si_{\be,I_{max}}}\tr(\pi_{(\be,l)}(g)\circ op_{F(\be,l)})\vert P_a(\be,l)\vert dl,\, g\in G,
 \end{eqnarray*}
where $P_a(\be,l) $ is the Pfaffian of the polynomial $Q(l)=\det{(\langle l,[Z_i,Z_j]_\be\rangle_{i,j\in I_{max}}}) $.
It follows from \cite{Lu-Z} that the function $f $ is Schwartz and the Fourier inversion theorem tells us that $\pi_{(\be,l)}(f)=op_{F(\be,l)} $ for any $(\be,l)\in\B\ti \g^* $.
\end{proof}

\section{Proof of the retract theorem}\label{proof}

The proof of Theorem \ref{retract} proceeds by induction on the length $\val I $ of the largest index set $I $ for which $(\B\ti \g^*)_I\cap M\ne \es $  and it will be done in several steps.

\subsubsection{ The case $I=\es $}

Suppose that all the elements $(\be,l)\in M $ are characters of $\g_\be $, which means that their index sets are empty.

Let us replace the variable group $(\B,G) $ by the group $(\CC,G)$, where $\CC=\B $ as a manifold, and the multiplications coming from $\CC $ are abelian, i.e. $[U,V]_\ga =0$ for every $U,V\in \g, \ga\in\CC $. We identify now the group $G $ with its Lie algebra and then $U\cdot_\ga V=U+V$ for every $U,V\in \g $ and $\ga\in\CC $.
This also means that $\chi_l$ is a character
on $G_{\ga}= \textrm{exp}_{\ga} \g$, for all $(\ga,l)
\in \CC \ti \g^*$. Now take ${\mathcal M}=M$.
Let $F \in {\mathcal S}(\R^r\ti M)$ be a kernel function with compact support in the variables $(\ga,l) $.
As $\R^r \ti M$ is a sub-manifold of $\R^r \ti \CC \ti \g^*$, the function $F$ may be extended to a
Schwartz function $\widetilde F$ (in the sense of Section \ref{function}
and \ref{kerfunction}) on $\R^r \ti \CC\ti \g^*$ with compact support in the variables $(\ga,l) $. Let
$f := (2\pi)^n{\mathcal F}^{-1}_3 \widetilde F$, where
${\mathcal F}^{-1}_3$ denotes the partial inverse Fourier
transform in the variable $l$ which is the third variable in $\R^r \ti \CC\ti \g^*$. Then $ f
\in {\mathcal S}(\R^r \times \CC \times \g^*)$. For all $(\alpha,(\ga,l)) \in \R^r \ti
{M}$, we have
\begin{eqnarray}
\nn \pi_{(\ga,l)} \big ( f(\alpha,\ga, \cdot) \big ) &=&
\widehat f^3(\alpha, \ga,l) \\
\nn &=& (2\pi)^n {\mathcal F}_3 {\mathcal F}_3^{-1}
F(\alpha,(\ga, l))\\
\nn & =&  F(\alpha, (\ga,l)).
\end{eqnarray}
In particular, $\pi_{(\ga,l)}(f(\alpha,\ga, \cdot)) =0$ if
$(\alpha, (\ga,l)) \in \R^r \ti (M\setminus \CC)$. The continuity
of the map $F \mapsto f$ is obvious. This proves the first  step in
the induction procedure.

\subsubsection{Reducing $\B $}\label{opensubs}

There are two cases where we can reduce the manifold $\B $.
\begin{enumerate}\label{}
\item
Suppose that there exists a smooth function $\va:\B\to \R_+ $ which is not constant on the subset $\B_M:=p_\B(M) $, where $p_\B:\B\ti \g^*\to \B $ is the canonical projection. Let $\be_0\in\B $ such that $\va(\be_0)\in ]a,b[$ for some  $ b>a>0 $ and let $\B_0:=\{\be\in \B; \frac{a}{2}< \va(\be)<2b\} $ and $M_0:=\{(\be,l)\in M; \be\in\B_0\} $.

Suppose that the theorem holds for the pair $(\B_0,M_0)$. Let us show that the result remains true for the pair $(\B,M) $. Let $\M_0 $ be an open relatively compact subset as in the theorem for $(\B_0,M_0 )$. We let $\M:=\{(\be,l)\in M; {a}<\va(\be)<b\}\cap \M_0 $. We will show that $\M $  works for $(\B,M )$. Note that since $\M_0 $ is open in $M_0 $, we have that $\M $ is open in $M $.

Let $F$ be a kernel
function defined on $\R^{r} \ti M \ti G \ti G$ such that
its support in $(\beta,l)$ is contained in $ G\cdot \M \subset M_0$.
By assumption, there exists
$f\in {\mathcal S}(\R^{r} \ti \B_0 \ti G)$ such that
$\pi_{(\beta,l)}(f(\cdot, \beta, \cdot))$ admits $F(\cdot,
(\beta,l),\cdot, \cdot)$ as an operator kernel if $(\beta,l) \in
M_0$. In particular, $\pi_{(\beta,l)}(f(\cdot, \beta,
\cdot))=0$ if $(\beta,l) \in M_0 \setminus G\cdot \M_0$. As $\B_0$
is a sub-manifold of $\B$, we may extend $f$ to a
function in ${\mathcal S}(\R^{r} \ti \B \ti G)$ which
we denote also by $f$.
Choose $\vartheta \in {\mathcal C}^{\infty}_c(\R)$ with
$0\leq \vartheta \leq 1$, $\vartheta \equiv 1$ on $[a,b]$ and
$\vartheta \equiv 0$ on $[0, \frac{a}{2}]\cup [2b,
+\infty[$. We define $\phi \in {\mathcal C}^{\infty}(M)$ by
$\phi(\beta,l):= \vartheta(\va(\be))$.
Then
$\phi\equiv 1$ on $G\cdot \M_0$ and $\phi \equiv 0$ on $M\setminus
G\cdot \M_0$. By taking $g:= \phi \cdot f$, we have that
$\pi_{(\beta,l)}(g(\cdot, \beta, \cdot))=\phi(\beta, l) \cdot
\pi_{(\beta,l)}(f(\cdot,\beta,\cdot))$.

If $(\beta,l) \in \M\subset \M_0 $, then $\pi_{(\beta,l)}(g(\cdot, \beta,
\cdot))=\pi_{(\beta,l)}(f(\cdot, \beta, \cdot))$ and it admits
$F(\cdot, (\beta,l), \cdot, \cdot)$ as an operator kernel. If
$(\beta,l) \in M_0\setminus G\cdot {\M_0}$, then $\pi_{(\beta,l)}(f(\cdot,
\beta, \cdot))=0$ and $\pi_{(\beta,l)}(g(\cdot, \beta, \cdot))=0$.
If $(\beta,l) \in M \setminus M_0$, then $\va(\beta)\in [0, \frac{a}{2}]\cup [2b,
+\infty[$, hence $\phi(\be,l)=0 $
and so $\pi_{(\beta,l)}(g(\cdot, \beta, \cdot))=0$.
Hence the result is true for the function $g$.
\item
If there exists a smooth sub-manifold $\B_0 $ of $\B $ such that $p_\B(M)\subset \B_0 $, then we can apply our theorem to the pair $(\B_0,M) $. Since every smooth function $f_0 $ on $\B_0\ti G $ can be extended to a smooth function $f $ on $\B\ti G $, the retract theorem also holds for  $(\B,M) $.
 \end{enumerate}

\begin{remark}\label{besubm}
\rm   Let $\B$ and $M$ be given as in the statement of the theorem. Let
\begin{eqnarray*}
 p_\B: M \to \B;\quad  p_\B(\be,l)=\be,
 \end{eqnarray*}
 be the canonical projection. If we denote by $M^{max} $ the subset of $M $ consisting of all $(\be,l)\in M $ for which the rank of $dp_\B(\be,l) $ is maximal, then $M^{max} $ is open in $M $ and the subset $p_\B(M^{max}) $ of $\B $ is a smooth sub-manifold of $\B $.
If $p_\B(M^{max}) $ contains at least 2 elements,  by the reasoning in Subsection \ref{opensubs}, using  a non-constant smooth function $\va_0 $ on  $p_\B(M^{max}) $, which can be extended to a smooth function $\va$ of $\B $, we reduce $\B $ to $\B^{max} $ and we can always assume in this way that $p_\B(M) $ is  a smooth sub-manifold of $\B $. If $p_\B(M^{max}) $ is a singleton $\{\be_0\} $, then  $M=M^{max}  $ and $p_\B(M) $ is  obviously a smooth sub-manifold of $\B $.
 \end{remark}

\subsection{Reducing to smoothly varying subspaces depending on $\B $}\label{smoothsub}

Let  $ M \subset \B\ti\g^*$ be a smooth $G $-invariant sub-manifold of $\B \ti \g^*$.
Let us fix the largest  index
\begin{eqnarray*}
 I_M=I=\big ( (j_1, k_1), \cdots, (j_r, k_r)\big )=(j_1,k_1)\times I_1,
 \end{eqnarray*}
where $I_1=((j_2,k_2),\cdots, (j_r,k_r))$ is the index set of $(\be, l\res{g^1(\be,\ell)}) $, such that the open subset $M_I:=(\B\ti \g^*)_I \cap M$ of $M $ is nonempty.
Let $p_\B:M\to \B; (\be, l)\mapsto \be $, be the projection onto the first variable and set
\begin{eqnarray*}
 \B_M:=p_\B(M),
 \end{eqnarray*}
which is  a smooth sub-manifold of $\B $ by Remark \ref{besubm}.

Let $ \cc_1:=\g_{j_{1}+1}= \textrm{span}\{Z_{j_1+1}, \dots, Z_n\}\subset \g$ and let
\begin{eqnarray}\label{nbeonnbondef}
\n_\be^1:= [\g,\cc_1]_\be+[Z_{j_1}, \g_{k_1+1}]_\be\subset \cc_1 \, \text{ for }\, \be\in \B.
\end{eqnarray}
Then, by the definition of the indices $(j_1, k_1) $, we have
\begin{eqnarray}\label{propnbonbbon}
 \n_\be^1\subset \ker {l}\cap\cc_1\subset \a_\be(l) \, \textrm{ if } \, (\be,l)\in (\B\ti \g^* )_{\leq I}.
\end{eqnarray}
It is easy to see  that $\n^1_{\beta}$ is an ideal in $\g$. Let
\begin{eqnarray*}
Z_\be:=[Z_{k_1},Z_{j_1}]_\be \,  \text{ for }\, \be\in \B.
\end{eqnarray*}

 We fix  a scalar product $ \langle{\cdot},{\cdot}\rangle $ on $ \g $ such that $\{Z_1, \dots, Z_n\}$ is an orthonormal basis
 and we identify $ \cc_1^* $ with $ \cc_1 $ by identifying $\sum_{r=j_1+1}^n a_rZ^*_r \in \cc_1^*$ with the element
 $\sum^n_{r=j_1+1} a_r Z_r$ of $\cc_1$. Denote by $\Vert \cdot \Vert_2$ the
 Euclidean norm on $\cc_1$ (and hence on $\cc_1^*$) with
 respect to the given scalar product.
We also identify $$(\n^1_{\beta})^{\perp} := \{q \in \cc_1^*; ~\langle q, \n^1_{\beta}\rangle =\{0\}\}$$
 with a subspace   of $\cc_1$.
For all $\beta \in \B$, we write $\cc_1
 = \n^1_{\beta} \oplus (\n^1_{\beta})^{\perp}$ and define $p_{\beta}$ to be the orthogonal projection of $\cc_1$ onto $(\n^1_{\beta})^{\perp}$.
For each $\beta \in \B$, a generating subset of $\n^1_{\beta}$ is given by
\begin{eqnarray*}
 V(\be)&=&\{v_1(\be),\cdots, v_s(\be)\}\\
&:=&\{[Z_a,Z_{a'}]_\be; a=1,\cdots, n, a'=j_{1}+1,\cdots, n\}\cup\{[Z_b,Z_{j_1}]_\be; b=k_1+1,\cdots,  n\}.
 \end{eqnarray*}
Let
\begin{eqnarray*}
 a_{j,j'}(\be):=\langle{v_j(\be)},{v_{j'}(\be)}\rangle \, \text{ for } \, {1\leq j,j'\leq s}.
 \end{eqnarray*}
Fix $ 0\leq k\leq s$, let $ \I_k=\{J\subset \{1,\cdots, s\}; \val
J=k\}$ and for $ \be\in\B$, let
\begin{eqnarray*}
 h_k(\be):= \sum_{J\in\I_k}\det{\big (\big(a_{j,j'}(\be)\big)_{j,j'\in J}}\big )^2.
 \end{eqnarray*}
It is easy to check that
\begin{eqnarray*}
h_k(\beta) \neq 0 &\Leftrightarrow& v_1(\beta), \dots, v_s(\beta)
\textrm{ have at least rank } k,\\
h_k(\beta)=0 &\Leftrightarrow& v_1(\beta), \dots, v_s(\beta)
\textrm{ have rank }r <k.
\end{eqnarray*}
Let $n_1\in\N $ and  put $f_0:= h_{n_1+1}$ and $f_1:= h_{n_1}$. Let
\begin{eqnarray*}
\B^{\leq n_1} &=& \{ \beta \in \B; \, f_0(\beta)=0\},\\
\B^{\geq n_1} &=& \{\beta \in \B; \, f_1(\beta)\neq 0\},\\
\B^{n_1} &=& \{\beta \in \B; \, f_0(\beta)=0 \textrm{ and } f_1(\beta) \neq 0\}.
\end{eqnarray*}
One sees that $\B^{\geq n_1}$ is open in $\B$, and hence is a
sub-manifold of $\B$.  Again, according to  the reduction argument in Subsection \ref{opensubs} we can assume that $f_1(\be)\ne 0 $ for all $\be\in \B $. On the other hand, let $ n_1:= \max_{\be\in \B_M}\dim{\n^1_\be}$, then we have
\begin{eqnarray}\label{maxdef}
 \nn\B^{\leq n_1}&:=&\{\be\in \B; \, \dim{\n^1_\be}\leq n_1 \},\\
 \nn\B^{n_1}&:=&\{\be\in \B; \, \dim{\n^1_\be}=n_1 \},\\
 \nn\B^{\geq n_1}&:=&\{\beta \in \B; \, \dim{\n^1_\be} \geq n_1\}.
\end{eqnarray}

Note that if we want $\n^1_{\beta}$ to be of fixed dimension and to have
$\n^1_{\beta}$, $(\n^1_{\beta})^{\perp}$ and $p_{\beta}$ to vary
smoothly with respect to $\beta$, we must restrict to $\B^{n_1}$. But in general $\B^{n_1}$ is not a sub-manifold of $\B$. Therefore we must find a smooth sub-manifold inside  $\B^{n_1}$ containing an open subset of the smooth manifold $\B_M=p_\B(M) $.
We have to distinguish the following two cases:

\medskip \noindent {\it Case 1:} If the differential $df_0$ is not identically zero
on $\B_{M} $, we may define
\begin{eqnarray*} \B^{max}_M:=\{\beta \in \B_M; \, df_0(\beta)\neq
0\} \, \, \text{ and } \, \B^{max}:= \{\beta \in \B; \, df_0(\beta)\neq 0\}.
\end{eqnarray*}
By assumption, $\B^{max}$ is a nonempty open subset of $\B$.

\medskip \noindent {\it Case 2:} Assume that $df_0$ is identically zero on $\B_{M}$.

If $\dim{ \B_M}< \dim \B $, we may build a function $\gamma \in
{\mathcal C}^{\infty}(\B)$ such that $\gamma\equiv 0$ on
$\B_{M}$ and $d\gamma$ is not identically zero on
$ \B_M$. We put $\widetilde f_0=f_0 +
\gamma$. Then $\widetilde f_0 \equiv 0$ on $\B_{M}$ and
$d\widetilde f_0$ is not identically zero on ${\B_M}$. We then define
\begin{eqnarray*}
\B^{max}_M:= \{ \beta \in \B_M; \, d\widetilde f_0(\beta)\neq 0\},\,
\B^{max}:= \{ \beta \in \B; \, \widetilde f_0(\be)=0, d\widetilde f_0(\beta)\neq 0\}.
\end{eqnarray*}
By the construction of $\widetilde f_0$, we have again that  $
 \B^{max} $ is an open subset of $\B $ and $\B^{max}_M\subset \B^{max} $.

If $\dim{ \B_M}=\dim\B $, then $ \B_M$ is open in $\B $ and we take a smooth function $\tilde f\ne 0 $ in $\B $ supported on $ \B_M $.
Let
\begin{eqnarray*}
 \B^{max}:= \{ \beta \in \B; \, \widetilde f(\beta)\neq 0\} \, \text{ and } \, \B^{max}_M:= \{ \beta \in \B_M; \,
 \widetilde f(\beta)\neq 0\}.
\end{eqnarray*}
\medskip
In the two cases, the ideals $ \n_\be^1 $
vary smoothly on the smooth sub-manifold $\B^{max}$ of $\B $, since $ \dim {\n_\be^1}=n_1$ on $\B^{max}_M$.
The projection $p_{\beta}$ also varies smoothly on $\B^{max}$.

\begin{remark}\label{nopbzb}
\rm   {According to Remark \ref{besubm}, we can now assume that $$\B=\B^{\rm max}. $$
Furthermore, since  the function $\be\mapsto \no{p_\be(Z_\be)}2^2 $ is now smooth on $\B $, we can take $\be^0\in\B $ and $0< \de< R< \iy$ such that  $\de< \no{p_{\be^0}(Z_{\be^0})}2< R$  and by using the reduction argument, we can then assume that the number $\no{p_\be(Z_\be)}2 $ is contained in the interval $[\de,R] $ for any $\be \in \B$. }
\end{remark}

\subsubsection{On the manifold $M$} \label{nmaxisbm}
Let us focus on the manifold $M$ again. Let
$(\beta_0, l_0) \in M$ be fixed, but arbitrary. There exist
$0< \delta< R< \infty$ such that
$$0<\delta<\min\{\val{\langle l_0, Z_{\beta_0} \rangle}, \Vert
p_{\beta_0} (Z_{\beta_0}) \Vert_2\} < \max\{\val{\langle l_0,
Z_{\beta_0} \rangle}, \Vert p_{\beta_0} (Z_{\beta_0}) \Vert_2\}< R.$$
This is due to the fact that $ M \subset (\B
\ti \g^*)_I$. According to Remark \ref{nopbzb} we can now assume that
$$0<\delta<\min\{\Vert
p_{\beta} (Z_{\beta}) \Vert_2\} < \max\{\Vert p_{\beta} (Z_{\beta}) \Vert_2\}< R$$
for all $\be\in \B $. We define
\begin{eqnarray*}
M^{\delta, R} &=& M^{ red}\\
&:=& \{(\beta,l)\in M; \, 0<\delta<\min\{\val{\langle l,
Z_{\beta} \rangle}\} <
\max\{\val{\langle l, Z_{\beta} \rangle}\} <R\}.
\end{eqnarray*}
Obviously, $M^{ red}$ is open in $M$
and thus is a smooth sub-manifold of $M$.
On the other hand, we define
\begin{eqnarray*}
(\B \ti \g^*)_{\leq I, R, \delta} &=& \Big\{(\beta,l) \in (\B \ti \g^*)_{\leq I}; \\
&{ }& \quad \frac{1}{2} \delta < \min\{ \val{ \langle l, Z_{\beta} \rangle}\}< \max\{
\val{ \langle l, Z_{\beta} \rangle} \} < \frac{3}{2}R\Big\}.
\end{eqnarray*}

\subsubsection{Reducing $M $}

 Now we show that if the retract theorem holds for
$(\B, M^{red})$, then it remains true for $(\B, M)$.

Assume that the result is true for $(\B, M^{ red})$.
Let ${\mathcal M}$ be the open subset in
$M^{ red}$ given by the assumption. We will show that one may take
the same manifold ${\mathcal M}$ for $(\B, M)$ such that the theorem remains true for $(\B,
M)$. As $M^{red}$ is open in $M$, the set ${\mathcal M}$ also has a nonempty interior in $M$.
Moreover, $p_{\B}(\mathcal M) \subset
p_\B(M^{red}) \subset \B^{n_1} \subset \B^{\geq n_1}$. Let
$\emptyset \neq C \subset {\mathcal M}$ be compact and let $F$ be
a kernel function defined on $\R^{r}\ti M\ti G\ti G$ whose support
in $(\beta,l)$ is contained in $G\cdot C$.
The restriction of $F$ to $\R^{r} \ti M^{red} \ti G \ti G$ is a kernel function for
$(\B, M^{red})$.

By assumption, there exists $f \in
{\mathcal S}(\R^{r} \ti \B \ti G)$ such that
$\pi_{(\beta,l)}(f(\cdot, \beta, \cdot))$ admits $F(\cdot,
(\beta,l), \cdot, \cdot)$ as an operator kernel if $(\beta,l) \in M^{red}$.
In particular, $\pi_{(\beta,l)}(f(\cdot, \beta,
\cdot ))=0$ if $(\beta,l) \in M^{red} \setminus C$.
 As $\emptyset \neq C \subset {\mathcal M}$ is
compact, there exist $\delta_1, R_1\in \R_+$ such that
\begin{eqnarray*} 0<\delta<\delta_1 &\leq& \min\{\val{\langle l,
Z_{\beta}\rangle }\}\\
&\leq &\max\{\val{\langle l,
Z_{\beta}\rangle }\}\\
&\leq& R_1<R
\end{eqnarray*}
for all $(\beta,l) \in C$, as $C \subset {\mathcal M} \subset M
\subset (\B \ti \g^*)_I$. Let $u\in {\mathcal C}_c^{\infty}(\R)$ be odd
such that $u\equiv 1$ on $[\delta_1, R_1]$ and $u\equiv 0$ on
$[0,\delta] \cup [R, +\infty[$. There exists $\chi
\in {\mathcal S}(\R)$ such that $\widehat {\chi}=u$. Let us define
a function $\psi$ on $\bigcup_{\beta}\{\beta\}\ti \exp{\R
Z_{\beta}}$ by $\psi(\beta, \exp{sZ_{\beta}}):= \chi(s)$. For
$(\beta,l) \in (\B \ti \g^*)_J$ with $J\leq I$, we have $Z_{\beta} \in
\a_{\beta}(l) \subset \g_{\beta}(l)$ and $\pi_{(\beta,l)}\vert_{\R
Z_{\beta}}(\psi(\beta, \cdot ))= \widehat {\chi}(\langle l,
Z_{\beta} \rangle )\Id_{\mathfrak H(\be,l)} = u( \langle l, Z_{\beta} \rangle )\Id_{\mathfrak H(\be,l)}$.

We define a function $g$ on $\R^r \ti \B\ti G$ by
$$g(\cdot, \beta, \cdot):= f(\cdot, \beta, \cdot ) * \psi(\beta, \cdot).$$
This implies that
$$\pi_{(\beta,l)}(g(\cdot, \beta, \cdot)) =
 u(\langle l, Z_{\beta}
\rangle ) \pi_{(\beta,l)}(f(\cdot, \beta, \cdot)).$$
If $(\beta,l)
\in C$, then  $u(\langle l,
Z_{\beta}\rangle )=1$ and $\pi_{(\beta,l)}(g(\cdot, \beta, \cdot))
= \pi_{(\beta,l)}(f(\cdot, \beta, \cdot))$ admits $F(\cdot,
(\beta,l), \cdot, \cdot)$ as an operator kernel.
If $(\beta,l) \in
M ^{red}\setminus C$, then $\pi_{(\beta,l)}(f(\cdot, \beta,
\cdot))=0$, hence $\pi_{(\beta,l)}(g(\cdot, \beta, \cdot))=0$ and
$F(\cdot, (\beta,l), \cdot, \cdot)=0$.
If $(\beta,l) \in M\setminus M^{red}$, then $\vert \langle l,Z_{\beta}
\rangle \vert \notin [\delta, R]$, i.e. $u(\langle l, Z_{\beta} \rangle
)=0$, which implies that
$\pi_{(\beta,l)}(g(\cdot, \beta, \cdot))=0$.
Hence, the mapping  $F\mapsto g$ satisfies the property of the retract for  $(\B, M)$.

\subsection{Construction of a new variable group}\label{constvar}

\subsubsection{The mapping $\al(\be,l) $}
For $(\beta,l) \in (\B\ti \g^*)_{\leq I}$, we have seen in (\ref{propnbonbbon}) that $\n^1_{\beta} \subset
\ker{l}$. Let  $q:=l\vert_{\cc_1} \in (\n^1_{\beta})^{\perp}$ and
$\langle l, Z_{\beta}\rangle= \langle l,
p_{\beta}(Z_{\beta})\rangle = \langle q,
p_{\beta}(Z_{\beta})\rangle$. For $(\beta,l) \in (\B\ti \g^*)_{\leq
I, R, \delta}$, we have that $\vert \langle l, [Z_{k_1}, Z_{j_1}]_{\beta}
\rangle \vert = \vert \langle l, Z_{\beta} \rangle \vert >
\frac{\delta}{2}>0$ implies $\langle l, [Z_{k_1},
Z_{j_1}]_{\beta} \rangle \neq 0$ and $j_1(\beta,l)=j_1$,
$k_1(\beta,l)=k_1$.

Take an odd function  $ \va\in C^\iy(\R) $  with the properties that $
\va(s)=0 $ for $0\leq  s<\de/4 \textrm{ and } s>2 R  $, $
1>\va(s)>0  $ for $ s\in ]\de/4,\de/2[ \cup ]3R/2,2R[  $ and $
\va(s)=1 $ for $ 3 R/2\geq s \geq \de/2 $.
For every $ (\be,q)\in \B\ti \cc_1^* $, we construct the vector $
\al(\be,q)\in \cc_1\simeq (\cc_1)^* \ $ by
\begin{eqnarray}\label{aldef}
\nn \al(\be,q)&:=&\va(\no{ p_\be(Z_\be)}2)\va({\langle{q},{p_\be(Z_\be)}\rangle}) p_\be(q) \\
 & & \quad +\big(1-\va(\no{ p_\be(Z_\be)}2)\va(\vert{\langle{q},{p_\be(Z_\be)}\rangle}\vert)\big)p_\be(Z_\be).
\end{eqnarray}
Then by the construction, $ \al(\be,q)\in (\n^1_\be)^\perp \subset
\cc_1^* \equiv \cc_1$ for every $ (\be,q)\in \B\ti \cc_1^*$.
On the other hand, for $(\be,l)\in (\B\ti \g^*)_{\leq I,R,\de}$ and $q=l\res{\cc_1} $, we have that
\begin{eqnarray}\label{uptosign}
 \nn\al(\be,q)&=&\va(\no{ p_\be(Z_\be)}2)\va({\langle{q},{p_\be(Z_\be)}\rangle}) p_\be(q)\\
 \nn& & \quad + (1-\va(\no{ p_\be(Z_\be)}2)\va(\vert{\langle{q},{p_\be(Z_\be)}\rangle}\vert))p_\be(Z_\be)\\
 \nn&=&  \pm p_\be(q)+(1-1)p_\be(Z_\be)\\
 &=& \pm q.
\end{eqnarray}
This is due to the fact that $p_{\beta}(q)=q$ as $\n^1_{\beta} \subset
\ker{q}$ for $q=l\vert_{\cc_1}$, if $(\beta,l) \in (\B\ti \g^*)_{\leq I}$.

We will show that
\begin{eqnarray*}
 \langle{\al(\be,q)},{Z_\be}\rangle
 &=&\va(\no{ p_\be(Z_\be)}2)\va({\langle{q},{p_\be(Z_\be)}\rangle})
 \langle{p_\be(q)},{Z_\be}\rangle\\
 &&\quad \quad +(1-\va(\no{ p_\be(Z_\be)}2)\va(\vert{\langle{q},{p_\be(Z_\be)}\rangle}\vert))\no{p_\be(Z_\be)}2^2> 0
\end{eqnarray*}
on $\B\ti \cc_1^* $. In fact, let us first notice that
$\langle p_{\beta}(q), Z_{\beta}\rangle = \langle q,
p_{\beta}(Z_{\beta})\rangle$. As $\varphi$ is an odd function
and $\varphi\geq 0$ on $\R_+$, we have
$$A:= \varphi(\Vert p_{\beta}(Z_{\beta})\Vert_2) \varphi(\langle
q, p_{\beta}(Z_{\beta})\rangle)  \langle p_{\beta}(q),
Z_{\beta}\rangle \geq 0.$$ Since $0\leq \varphi \leq 1$ on $\R_+$,
$$B:= \big (1- \varphi(\Vert p_{\beta}(Z_{\beta})\Vert_2) \varphi
(\vert \langle q, p_{\beta}(Z_{\beta})\rangle \vert ) \big ) \Vert
p_{\beta}(Z_{\beta})\Vert^2_2 \geq 0.$$ If none of the
$\varphi(\cdot)$'s is equal to zero and if $\langle p_{\beta}(q), Z_{\beta}\rangle\ne 0$, then $A>0$.
If $\langle p_{\beta}(q), Z_{\beta}\rangle=0$, then $\varphi(\vert\langle p_{\beta}(q), Z_{\beta}\rangle\vert)=0 $ and so $ B>0$, as now by Remark \ref{nopbzb} $\Vert p_{\beta}(Z_{\beta})\Vert_2>0$.
If one of the $\varphi(\cdot)$'s is equal to zero, then again $B>0$.

For $(\beta,q)\in \B\ti \cc_1^*$,  let
\begin{eqnarray*}
\de(\be,q):=\ad_\be^*(Z_{j_1})\al(\be,q)\in\g^*.
\end{eqnarray*}
We have that
\begin{eqnarray*}
\langle{\de(\be,q)},{Z_{k_1}}\rangle=\langle{\al(\be,q))},{Z_\be}\rangle> 0
\end{eqnarray*}
 and
 \begin{eqnarray*}
\langle{\de(\be,q)},{[\g,\g] _\be}\rangle=\langle{\al(\be,q)},{[[\g,\g]_\be,Z_{j_1}] _\be}
\rangle\subset \langle{\al(\be,q)},{[\g,[\g,Z_{j_1} ]_\be] _\be}\rangle\\
\subset \langle{\al(\be,q)},{[\g,\cc_1] _\be}\rangle\ \subset
\langle{\al(\be,q)},{\n^1_\be}\rangle=\{0\},
\end{eqnarray*}
by the definition of $\alpha(\beta,q)$ in \eqref{aldef}. This means that
$\delta(\beta,q)$ is an algebra homomorphism of $\g_\be=(\g, [ \cdot, \cdot]_{\beta})$ which does not vanish at the vector $Z_{k_1} $. Hence the
subspace
\begin{eqnarray*}
\g^1(\be,q):=\ker{\de(\be,q)}
\end{eqnarray*}
is an ideal of $ \g_\be$ of co-dimension 1 and
\begin{eqnarray}\label{decompo}
 \g=\R Z_{k_1}\oplus \g^1(\be,q).
 \end{eqnarray}
Furthermore $\g^1(\be,q) $ contains $ \cc_1 $ for any $ (\be,q) \in \B\ti \cc_1^*$. In fact,
$$\langle \delta(\beta,q),\cc_1\rangle = \langle \alpha(\beta,q),
[\cc_1, Z_{j_1}]_{\beta}\rangle =0 $$ as $\alpha(\beta,q) \in
(\n^1_{\beta})^{\perp}$ and $[\cc_1, Z_{j_1}]_{\beta} \subset \n^1_{\beta}$.

\subsubsection{ The new variable group $(\B_1,G_1) $}
In order to construct a new variation in the induction procedure, we put
\begin{eqnarray*}
\B_1:=  \B\ti\R\ti \cc_1^*.
\end{eqnarray*}
For $(\be,y,q)\in\B_1 $, we define a Jordan-H\" older basis
\begin{eqnarray*}
 \tilde{\ZZ}^1(\be,y,q)=\{\tilde Z^1_1(\be,q),\cdots, \tilde Z^1_{{n-1}}(\be,q)\}
\end{eqnarray*} of
$\g_1(\beta,y,q)=\ker{\delta(\beta,q)}$ by
\begin{eqnarray*}
\al^{\be,y,q}_{k}=\al^{\be,q}_{k}:=\frac{\langle{\al(\be,q)},{[Z_k,Z_{j_1}]_\be
}\rangle}{\langle{\al(\be,q)},{Z_\be}\rangle}
\end{eqnarray*}
and
\begin{eqnarray*}
\widetilde{\ZZ}^1(\be,y,q)=\widetilde{\ZZ}^1(\be,q)&:=&\{ Z_1-\al^{\be,q}_{1} Z_{k_1}, ,\cdots, Z_{k_1-1}-\al^{\beta,q}_{k_1-1} Z_{k_1}, Z_{k_1+1},\cdots, Z_n\}\\
&=&\{\tilde Z^1_1(\be,q),\cdots, \tilde Z^1_{n-1}(\be,q)\}.\
\end{eqnarray*}
In particular, for $(\be,l)\in (\B\ti\g^*)_{\leq I, R, \delta} $ we have by
(\ref{gonbel}) that
\begin{eqnarray*}
 \g_1(\beta ,y,l\vert_{\cc_1}) = \g_1(\beta,l)=\g^1(\be, l\res{\cc_1}).
 \end{eqnarray*}
In fact, in this case $j_1(\beta,l) = j_1$, $k_1(\beta,l)=k_1$ and
\begin{eqnarray*}
 \g_1(\beta,y, l\vert_{\cc_1}) = \{U \in \g~\vert~\langle
\delta(\beta, l\vert_{\cc_1}),U\rangle=0\}=\{U\in \g~\vert~\langle
l, [U,Z_{j_1}]_{\beta}\rangle =0\} = \g_1(\beta,l),
 \end{eqnarray*}
as $\alpha(\beta, l\vert_{\cc_1})=\varepsilon \cdot l\vert_{\cc_1}$
with $\varepsilon=\pm 1$ if $(\beta,l)\in (\B \ti \g^*)_{\leq I, R, \delta}$.

For each $k$, we also have that
\begin{eqnarray}\label{laequ}
 \al_k^{\be,q}=\frac{\langle{l,[Z_k,Z_{j_1}]_\be
}\rangle}{\langle{l},{Z_\be}\rangle}.
 \end{eqnarray}
 This new basis $\widetilde{\ZZ}^1(\be,y,q) $
 coincides then, up to normalisation, with the basis obtained in Section
 \ref{step}, both procedures and bases generate the same indices.
 Furthermore by \eqref{laequ}, for $(\be,l)\in
(\B\ti\g^*)_{\leq I,R,\de} $, we have
\begin{eqnarray}\label{againeqone}
 \tilde\ZZ^1(\be,y,l_1)=\ZZ^1(\be,l),
 \end{eqnarray}
where $\ZZ^1(\be,l) $ is defined in Section \ref{step} and $l_1= l\vert_{\cc_1}$.

For any $(\be, y, q) \in \B_1$, let us write
\begin{eqnarray*}
[\tilde Z^1_u(\be,q),\tilde Z^1_v(\be,q)] =\sum_{w=1}^{n-1}
\ga(\be,q)^{u,v}_w \tilde Z^1_w(\be,q) \quad \textrm{for } u<v
\textrm{ in }\{1,\cdots,n-1 \}.
\end{eqnarray*}
We obtain in this way a new variable Lie algebra $ (\B_{1}, \g_1)$, where
\begin{eqnarray*}
 \g_1=\R^{n-1},\ \B_1=\B\ti\R\ti\cc_1^*
 \end{eqnarray*}
 and the canonical basis $ \ZZ^1=\{Z^1_1,\cdots, Z^1_{n-1}\} $ of $\g_1$ satisfies, by definition,
\begin{eqnarray*}
[Z^1_u,Z^1_v] _{(\be,q)}=\sum_{w=1}^{n-1} \ga(\be,q)^{u,v}_w
Z^1_{w}, \quad \textrm{ for } u<v \textrm{ in } \{1,\cdots,n-1 \}.
\end{eqnarray*}
This means that the new variable Lie algebra $(\B_1, \g_1)$ with $\g_1 \equiv
\g^1(\beta,q)$ is defined such that $(\g_1, [\cdot,
\cdot]_{(\beta,q)}) \equiv (\g^1(\beta,q), [\cdot, \cdot]_{(\beta,q)})$.

Given $(\beta,l) \in \B\ti \g^*$, let us define
$l_1 \in \g_1^*$ by $l_1(Z^1_i) := l( \tilde Z^1_i(\be,q))$ for all $i\in \{1, \dots, n-1\}$. One has
$l_1(Z^1_i)= l(Z_{i+1})$ if $i \geq k_1$. We also define a map
\begin{eqnarray}\label{iota}
\nn \iota_1:\B\ti \g^*&\rightarrow& \B_1 \ti \g_1^*\\
(\beta,l) &\mapsto& ( (\beta, \langle l,Z_{k_1}\rangle ,\al(\be, l\res{\cc_1})),  l_1),
\end{eqnarray}
where $l_1\equiv l\res{\g^1(\be,l\res{\cc_1})} $. We see that $ \iota_1$ is obviously smooth, injective and even a diffeomorphism onto its image.

Using (\ref{decompo}) we can identify every $l\in \g^* $ with
the pair $(v,l_1) $ where $v:=\langle \ell ,Z_{k_1}\rangle  $ and $l_1:=l\res{\g_1}\equiv l\res{\g^1(\be,l\res{\cc_1})} $.
We can then transfer the natural action of $G $ on $\B\ti \g^* $ to $\B_1\ti \g_1^*$ using the mapping $\iota_1 $. This gives us
 \begin{eqnarray*}
 g\cdot ((\be,v,q),l_1)= ((\be,v+\langle \Ad^*_\be(g)l_1,Z_{k_1}\rangle,\Ad_\be^*(g)q),\Ad^*_\be(g)l_1).
 \end{eqnarray*}
Then we have automatically the relation
\begin{eqnarray*}
 \iota_1(g\cdot (\be,l))=g\cdot(\iota_1(\be,l))
 \end{eqnarray*}
for any $g\in G $ and $(\be,l)\in\B\ti\g^* $.

Consider now the smooth manifold
\begin{eqnarray*}
 (\B\ti\g^*)^0_{\leq I, R, \de} :=\{(\be,l)\in (\B\ti\g^*)_{\leq I,R,\de}; \langle l,Z_{k_1}\rangle=\langle l,Z_{j_1}\rangle=0\}.
 \end{eqnarray*}
Then obviously the smooth manifold $(\B\ti\g^*)_{\leq I,R,\de} $ is diffeomorphic with the manifold $\R^2\ti (\B\ti\g^*)^0_{\leq I,R,\de} $.
The mapping
\begin{eqnarray*}
  \Phi:\R^2\ti (\B\ti\g^*)^0_{\leq I,R,\de} \to (\B\ti\g^*)_{\leq I,R,\de}
 \end{eqnarray*}
 given by
\begin{eqnarray*}
 \Phi(s,t,(\be,l)):=\Big(\be, \Ad^*(\exp{\frac{s}{\langle l,Z_\be\rangle } Z_{j_1}}\exp{\frac{t}{\langle l,Z_\be\rangle} Z_{k_1}})l\Big)
 \end{eqnarray*}
is such a diffeomorphism.
Hence every smooth $G $-invariant sub-manifold $M $ of $(\B\ti\g^*)_{\leq I,R,\de} $ can be decomposed into a direct product of $\R^2 $ with the smooth manifold $M^0 $, where
\begin{eqnarray*}
 M^0:=\{(\be,l)\in M; \langle l,Z_{k_1}\rangle=\langle l,Z_{j_1}\rangle= 0\}.
 \end{eqnarray*}

For $(\beta,l) \in \B\ti \g^*$, one has
$l_1(Z^1_i)=l(Z_i)$ if $i<k_1$ and $l_1(Z^1_i)=l(Z_{i+1})$ if $i\geq k_1$.
We remark that for $(\be,l)  $ and $(\be,l')  $ in $M $ with
$\iota_1(\be,l)=\iota_1(\be,l') $ we have that $l $ and $l' $ have the same restriction to $\g_1(\be,l)=\g_1(\be,l') $, so they are  on the same co-adjoint orbit and $l'=\Ad^*(y)l $ for some $y\in P(\be,l) $ and hence
\begin{eqnarray}\label{opeqq}
 \nn F(\be,l)=F(\be,l')
\end{eqnarray}
by the conditions on the operator fields defined over $M $ given in Definition \ref{kernelf}.

We denote the  new variable Lie group by $\mathbb G_1=(\B_1, G_1)$
where $G_1=(\textrm{exp}_{\beta_1} \g_1)_{\beta_1 \in \B_1}$ and
$\textrm{exp}_{\beta_1} \g_1$ is the connected, simply connected,
nilpotent Lie group associated to the Lie algebra $(\g_1, [\cdot, \cdot ]_{\beta_1})$.

\subsection{Induction step}

\medskip \noi To simplify notations, from now on we will omit the
subscript $\beta$ in the notations of the multiplication and the
exponential map, unless the subscript is crucial for the
understanding. There are two preliminary steps to check.

\subsubsection{Induction hypothesis }

\noi In this subsection, we
will prove the result for $(\B, M)$ using induction.
Let $M_1=\iota_1(M) $ be as constructed in \eqref{iota}.
Let us recall that $l_1=l_1(\beta, l\vert_{\cc_1}) \equiv
l\vert_{\g_1(\beta, l\vert_{\cc_1})}$ for $ (\be,l)\in
(\B\ti\g^*)_{\leq I} $. The Vergne polarisation
$\p(\beta,l)$ for $l$ in $(\g,[\cdot, \cdot]_{\beta})$, obtained
by the procedure of Ludwig-Zahir (see \cite{Lu-Z}, \cite
{Lu.Mo.Sc.}), is also the Vergne polarisation for $l_1$ in $(\g_1, [\cdot, \cdot]_{\beta})$. Let us denote by
$P(\beta,l)=\textrm{exp}_{\beta} \p(\beta,l))$ the corresponding
subgroup. The associated induced representations will be denoted
by $\pi_{(\beta,l)} := \text{ind}^G_{P(\beta,l)} \chi_l$, respectively, $\tilde \pi_{((\beta, l\vert_{\cc_1}),
l_1)}:=\text{ind}^{G_1}_{P(\beta,l)}
\chi_{l_1}$. Then $\pi_{(\beta,l)} \cong \text{ind}^G_{G_1} \tilde \pi_{((\beta, l\vert_{\cc_1}), l_1)}$, as
usual.

Since $M $ is $G $-invariant, the manifold $M_1=\iota_1(M) $ is also $G$-invariant in $\B_1\ti \g_1^* $. Hence we can write $M_1 $ as a direct product manifold $\R^2\ti M_1^0 $, where
\begin{eqnarray*}
 M_1^0:=\{((\be, 0, q),l_1); \langle l_1,Z_{j_1}\rangle=0\}
 \end{eqnarray*}
is $G_1 $ invariant. Note that  $M_1 ^0$ is contained in $(\B_1\ti\g^*_1 )_{I_1}$ and for every $((\be,v,l\res{\cc_1}),l_1)\in M_1 $ we have that $\iy> R>\vert \langle l_1,Z_\be\rangle\vert >\de>0 $.
The induction hypothesis in $\B_1\ti\g_1^* $ and $M_1^0
\subset (\B_1\ti \g_1^*)_{I_1} $ gives us an open relatively compact non-empty  subset  ${\mathcal
M}_1^0$ of  $M_1^0$ with the required properties of the theorem.

 We choose now a relatively compact open subset $\M_1 $ of $M_1 $ such that $\ol{\M_1}\subset  M_1$ and $\M_1 $ is contained in $G\cdot \M_1^0 $. Let
\begin{eqnarray*}
 \M:=\iota_1\inv(\M_1) \quad \text{and} \quad \M^0:=\iota_1\inv(\M_1^0).
 \end{eqnarray*}
Then $\M $ is  non-empty open with its closure $\ol\M $ contained in $M$ and $\M $ is contained in $G\cdot \M_0 $.
We take  a  kernel function $ F\in {\mathcal D}^c_{M}$ such that its support  is contained in $\R^r\ti G\cdot \M\ti G\ti G $.

Given the kernel function $F$, we will now define a kernel
function for the variable group $(\B_1, G_1)$. For simplicity, we
will omit the subscripts $\beta$ or $(\beta, v,l\vert_{\cc_1})$ in
the notations of the multiplication and the exponential map, and
we will identify $g_1, g_1'\in \G_1 =\big ((G_1)_{(\beta_1)}\big
)_{\be_1\in {\mathcal B}_1} \equiv \g_1$ with the
corresponding elements in $ G_1$.
In the following computations, the parameters $\be$ and $(\beta,v, l\vert_{\cc_1})$
will indicate how to multiply group elements or
how to decompose smoothly the group elements, even if it is not
marked explicitly. For $\iota_1(\be,l)=((\be,\langle l,Z_{k_1}\rangle,l\res{\cc_1}), l_1)\in M_1 $, we put
\begin{eqnarray*}
 & &F_1(\alpha, u, t, ((\be,\langle l,Z_{k_1}\rangle,l\res{\cc_1}), l_1), g_1, g'_1) := \\
& & \quad \quad (2
\pi)^{n-j_1+1} \cdot \vert c(\beta, l) \vert\cdot
F(\alpha, (\be,l), \exp{(u+t)X} \cdot g_1, \exp{tX} \cdot g'_1),
 \end{eqnarray*}
for $\al\in\R^r, u,t\in\R$ and $g_1,g_1'\in G_1 $, where
$c{(\beta,l)}:= \langle l, [Z_{k_1}, Z_{j_1}]_{\beta}\rangle \ne 0$ and $X=Z_{k_1}$.
This function $F_1$ has its support
$S_1:=\iota_1(S)$ contained in $G\cdot\M_1 $, and belongs to  $ \D^c_{M_1} $.
The operator field $F_1 $ is smooth on $M_1 $, since  the  mappings $F$ and $c$ are both smooth.

\noi By the induction hypothesis, there exists $h \in {\mathcal
S}( \R^{r+2}, {\mathcal B}_1, G_1)$ such that $\tilde
\pi_{((\beta_1, l\vert_{\cc_1}), l_1)} \big ( h(\alpha, u, t,
\beta_1, \cdot )\big )$ admits $F_1(\alpha, u, t,
(\be_1, l_1), \cdot, \cdot)$ as an operator
kernel for all $(\beta_1,  l_1) \in {M}_1^0$. The
construction of the retract function $f$ will now be done in
several steps.

\subsubsection { Definition of the retract function on the
original group}

For $ (\be, v, q)\in \B_1$, let us first define $\widetilde h$ by
\begin{eqnarray*}
\widetilde h(\alpha, u, t, (\be,v,q), g_1)&:= &\int_{\R} \int_{\cc_1}
h(\alpha, u, t, (\be,v,q), g_1 \cdot \exp{yY} \cdot \exp{Z}) e^{-iq(Z)} dZ dy,
\end{eqnarray*}
where $Y=Z_{j_1}$ and $Z=Z_\be=[X,Y]_\be $ with $X= Z_{k_1}$.
The integral converges, as $h$ is Schwartz in $\g_1$ (for fixed
$\beta_1$), and it is of rapidly decreased in $ q\in(\cc_1)^* $,
because it is  a Fourier transform in $Z $. For all
$(\beta,v,q) \in \B \ti\R\ti \cc_1^*$, we then define
\begin{eqnarray} \nn \widetilde f(\alpha, (\beta,v,q) , \exp{uX} \cdot g_1 \cdot
\exp{yY} \cdot \exp{Z}) = e^{iq(Z)} \int_{\R} \widetilde h(\alpha,
u, t, (\beta,v,q), g_1^{-t}) e^{-ityq([X,Y]_{\beta})}dt
\end{eqnarray}
with $g= \exp{uX} \cdot g_1$ and $g_1^{-t}:= \exp{-tX} \cdot g_1 \cdot \exp{tX}$.
The function $\widetilde f$ is smooth on $\R^r
\times (\B\ti \cc_1^*)\times G$.
As $\widetilde f$ is of rapid decrease
 in $q\in \cc_1^*$ by construction, we may define $f$ by
$$f(\alpha,\beta, g):= \int_{(\cc_1)^*} \widetilde f(\alpha, (\beta,0, q),g)dq, \quad \al\in\R^r, \be\in\B, g\in G.$$
One can see that $f\in {\mathcal S}(\R^r, {\mathcal B}, G)$ (in the sense of Section \ref{function}).

\subsubsection { The retract property}

Let us now compute $\pi_{(\beta,l)}\big
(f(\alpha,\beta,\cdot) \big )$ for $(\beta,l)\in M $. Since the manifold $M $ is contained in $(\B\ti \g^*)_{\leq I,R,\de} $ we have that  $\cc_1 \subset \a_{\beta}(l)
\subset \g_{\beta}(l)$. If we
identify $\exp{\cc_1}$ and $\cc_1$, as well as $\exp{Z}$ and $Z$, for any function $\xi(\beta) \in {\frak
H}_{(\beta,l)}$ (the representation space of $\pi_{(\beta,l)}$)
and any $\tilde g \in G$, we have that
\begin{eqnarray}\label{}
\nn \Big( \pi_{(\beta,l)} \big (f(\alpha,\beta,\cdot) \big)
\xi(\beta) \Big )(\tilde g) &= &\int_{G_{\beta}/\cc_1} \int_{\cc_1}
f(\alpha, \beta, g \cdot Z) \Big ( \pi_{(\beta,l)}(g)
\pi_{(\beta,l)} (Z) \xi(\beta) \Big
)(\tilde g) dZ d\dot g\\
\nn &= &\int_{G_{\beta}/\cc_1} \int_{\cc_1} f(\alpha,\beta, g \cdot Z)
e^{-il(Z)} \Big ( \pi_{(\beta,l)} (g) \xi(\beta) \Big ) (\tilde g) dZ d\dot g\\
\nn &= &  \int_{G_{\beta}/\cc_1} \int_{\cc_1} \int_{(\cc_1)^*}
\tilde f(\alpha, (\beta,0,q), g \cdot Z) e^{-il(Z)} \Big
(\pi_{(\beta,l)}(g) \xi(\beta) \Big)
(\tilde g) dq dZ d\dot g\\
\nn &= &  \int_{G_{\beta}/\cc_1} \int_{\cc_1} \int_{(\cc_1)^*}
\tilde f(\alpha, (\beta,0,q),g) e^{iq(Z)} e^{-il(Z)} \Big (
\pi_{(\beta,l)}(g) \xi(\beta)
\Big ) (\tilde g) dq dZ d\dot g\\
\nn &= & \Big ( \frac{1}{2\pi}\Big )^{n-j_1}\int_{G_{\beta}/\cc_1}
\tilde f(\alpha, (\beta, 0, l\vert_{\cc_1}), g) \Big (\pi_{(\beta,l)}
(g) \xi(\beta) \Big) (\tilde g) d\dot g.
\end{eqnarray}
We use the following smooth decomposition:
$X=Z_{k_1}, \g_1=\g_1(\be, l) $ which gives us
$$\tilde g = \exp{sX} \cdot \tilde g_1 \quad \text{with} \quad
s=s(\tilde g, \beta, l\vert_{\cc_1}), \tilde g_1= \tilde g_1(\tilde
g, \beta, l\vert_{\cc_1}).$$
We then obtain (using the fact that $ \cc_1\subset (\a_\be(l))  $ for all  our $ (\be,l) $'s) that:
\begin{eqnarray*}\label{ } \nn { }&&(2\pi)^{n-j_1} \cdot \Big (
\pi_{(\beta,l)} \big (f(\alpha,\beta, \cdot) \big )\xi(\beta)
\Big )(\exp{sX} \cdot \tilde g_1) \\
\nn &&=~ \int_{G_{\beta}/\cc_1} \tilde f(\alpha, (\beta, 0, l\vert_{\cc_1}), g) \Big (\pi_{(\beta,l)}(g)\xi(\beta) \Big
)(\exp{sX} \cdot
\tilde g_1)d\dot g\\
\nn &&=~ \int_{\R} \int_{G_1/\cc_1} \tilde f(\alpha, (\beta, 0, l\vert_{\cc_1}), \exp{uX}\cdot g_1)\cdot \xi(\beta)(g_1^{-1} \cdot
\exp{(s-u)X} \cdot \tilde g_1) d\dot g_1 du\\
\nn &&=~ \int_{\R} \int_{G_1/\cc_1} \tilde f(\alpha, (\beta, 0, l\vert_{\cc_1}), \exp{uX} \cdot g_1)\cdot  \\
\nn &&  \quad   \tilde \xi(\beta, 0,
l\vert_{z})(s-u)\Big (\big (\exp{-(s-u)X} \cdot g_1 \cdot \exp{(s-u)X} \big )^{-1} \cdot\tilde g_1 \Big) d\dot g_1 du\\
\nn && \text{with } \tilde \xi(\beta,0,l\vert_{\cc_1})(v)(g_1):= \xi(\beta)(\exp{vX}\cdot g_1)\\
\nn &&  =~ \int_{\R} \int_{G_1/\cc_1} \tilde f(\alpha,(\beta, 0, l\vert_{\cc_1}), \exp{(s-r)X}\cdot g_1) \cdot \\
\nn && \quad \tilde \xi(\beta,0,l\vert_{\cc_1})(r) \Big (\big (\exp{-rX} \cdot g_1 \cdot \exp{rX} \big )^{-1} \cdot \tilde g_1 \Big)d\dot g_1 dr \quad \text{with } s- u= r\\
\nn &&  =~ \int_{\R} \int_{G_1/\cc_1} \tilde f(\alpha,
(\beta,0, l\vert_{\cc_1}), \exp{(s-r)X} \cdot g_1^r) \cdot
\tilde \xi(\beta,0,l\vert_{\cc_1})(r)(g_1^{-1} \cdot \tilde g_1) d\dot g_1 dr\\
\nn && \text{with }g_1^r=\exp{rX} \cdot g_1 \cdot \exp{-rX}\\
\nn &&   =~ \int_{\R} \int_{G_1/\cc_1} \tilde f(\alpha, (\beta, 0, l\vert_{\cc_1}), \exp{(s-r)X} \cdot g_1^r) \Big (\tilde
\pi_{((\beta, 0, l\vert_{\cc_1}), l_1)}(g_1)\tilde
\xi(\beta,0,l\vert_{\cc_1})(r) \Big ) (\tilde g_1) d\dot g_1 dr\\
&&(\text{with } l_1\equiv l\vert_{\g_1})\\
\nn &&   =~ \int_{\R} \int_{G_1/\cc_1} \int_{\R} \tilde h(\alpha,
s-r, t, (\beta, 0, l\vert_{\cc_1}), g_1^{r-t}) \Big ( \tilde
\pi_{((\beta, 0, l\vert_{\cc_1}), l_1)}(g_1) \tilde
\xi(\beta,0,l\vert_{\cc_1})(r) \Big )(\tilde g_1) dt d\dot g_1 dr \\
\nn &&   =~ \int_{\R} \int_{G_1/\exp{\R Y}\cdot \cc_1} \int_{\R}
\int_{\R} \tilde h(\alpha, s-r, t, (\beta, 0, l\vert_{\cc_1}), w_1^{r-t}\cdot (\exp{yY}^{r-t}) \\
\nn &&   \quad    \Big ( \tilde
\pi_{((\beta, 0, l\vert_{\cc_1}), l_1)} (w_1)\tilde
\xi(\beta,0,l\vert_{\cc_1})(r) \Big )(\tilde g_1) dt dy d\dot w_1 dr\ (
 \text{as } l(Y)=0,\text{ for } l_1
\equiv l\vert_{\g_1(\beta, 0, l\vert_{\cc_1})})\\
\nn &&   =~\int_{\R} \int_{G_1/\exp{\R Y}\cdot \cc_1} \int_{\R}
\int_{\R} \tilde h(\alpha, s-r,t, (\beta, 0, l\vert_{\cc_1}), w_1^{r-t}
\cdot \exp{yY}) \cdot e^{i(r-t) c(\beta,  l)y}\\
\nn &&  \quad     \Big ( \tilde
\pi_{((\beta, 0, l\vert_{\cc_1}), l_1)} (w_1) \tilde
\xi(\beta,0,l\vert_{\cc_1})(r) \Big )(\tilde g_1) dtdyd\dot w_1 dr\\
\nn && (\text{as }\big (\exp{yY}\big)^{r-t}=\exp{yY + y(r-t)[Z_{k_1}, Z_{j_1}]_{\beta}}, \, \text{by covariance of }  \tilde h, \text{ with } \\
\nn &&   c(\beta, l)= \langle l, [Z_{k_1}, Z_{j_1}]_{\beta}\rangle \neq 0 \text{ as before})\\
 \nn &&   =~\frac{1}{\vert
c(\beta, l)\vert} \int_{\R} \int_{G_1/\exp{\R Y}\cdot
\cc_1} \int_{\R} \int_{\R} \tilde h(\alpha, s-r, t, (\beta, 0, l\vert_{\cc_1}), w_1^{r-t} \cdot \exp{c(\be,l)^{-1} \tilde yY} )
e^{ir\tilde y} e^{-it \tilde y} \\
\nn &&   \quad    \Big ( \tilde \pi_{((\beta, 0, l\vert_{\cc_1}), l_1)}
(w_1) \tilde \xi(\beta,0,l\vert_{\cc_1})(r) \Big ) (\tilde g_1) dt
d\tilde y d\dot w_1 dr \quad  (\text{with } \tilde y = c(\be,l) y)\\
\end{eqnarray*}
\begin{eqnarray*}
\nn &&   =~\frac{1}{\vert c(\beta, l)\vert} \int_{\R}
\int_{G_1/\exp{\R Y}\cdot \cc_1} \int_{\R} \int_{\R}
\int_{\R} \int_{\cc_1}h(\alpha, s-r, t, (\beta, 0, l\vert_{\cc_1}), w_1^{r-t} \\
&&  \quad \cdot \exp{c(\beta, l)^{-1} \tilde yY} \cdot
\exp{y'Y} \cdot Z)\cdot e^{-il(Z)} e^{ir\tilde y} e^{-it \tilde y}\\
\nn &&  \quad \Big ( \tilde \pi_{((\beta, 0, l\vert_{\cc_1}), l_1)}(w_1)
\tilde \xi(\beta,0,l\vert_{\cc_1})(r) \Big ) (\tilde g_1) dZ dy' dt d\tilde y d\dot w_1 dr\\
\nn &&   =~\frac{1}{\vert c(\beta, l)\vert} \int_{\R}
\int_{G_1/\exp{\R Y}\cdot \cc_1} \int_{\R} \int_{\R} \int_{\R}
\int_{\cc_1} h(\alpha, s-r, t, (\beta, 0, l\vert_{\cc_1}), w_1^{r-t} \cdot \exp{y''Y} \cdot Z)\\
\nn &&  \quad     \cdot e^{-il(Z)}
e^{ir\tilde y} e^{-it \tilde y} \Big ( \tilde \pi_{(\be_1, l_1)}(w_1) \tilde \xi(\beta,0,l\vert_{\cc_1})(r) \Big )
(\tilde g_1) dZ dy'' dt d\tilde y d\dot w_1 dr\   (\text{for } y''= y'+ c(\beta, l)^{-1} \tilde y)\\
\nn &&   =~\frac{1}{2\pi} \cdot\frac{1}{\vert c(\beta, l)\vert} \int_{\R} \int_{G_1/\exp{\R Y}\cdot \cc_1}
\int_{\R} \int_{\cc_1} h(\alpha, s-t, t,(\beta, 0, l\vert_{\cc_1}), w_1 \cdot \exp{y''Y} \cdot Z)\\
\nn &&  \quad     \cdot e^{-il(Z)} \Big ( \tilde \pi_{((\beta, 0, l\vert_{\cc_1}), l_1)}(w_1) \tilde
\xi(\beta,0,l\vert_{\cc_1})(t) \Big ) (\tilde g_1) dZ dy'' dt d\dot w_1\\
\nn &&   =\frac{1}{2\pi}\cdot \frac{1}{\vert c(\beta, l)\vert} \int_{\R} \tilde
\pi_{((\beta, 0, l\vert_{\cc_1}), l_1)} \big ( h(\alpha, s-t, t,
(\beta, 0, l\vert_{\cc_1}), \cdot ) \tilde \xi(\beta,0,l\vert_{\cc_1})(t) \big ) ( \tilde g_1) dt\\
\nn &&  \text{ as } e^{-il(Z)} \tilde \pi_{((\beta, 0, l\vert_{\cc_1}), l_1)} (w_1) =
\tilde \pi_{((\beta, 0, l\vert_{\cc_1}), l_1)} (w_1 \cdot \exp{y''Y}\cdot Z).
\end{eqnarray*}

\medskip
Let us finish the computation for $(\beta, l) \in M$. It suffices  to take $(\be,l)\in  M^{0}$. Then
$((\beta, 0, l\vert_{\cc_1}), l_1) \in M_1^0$ and by the induction hypothesis,
\begin{eqnarray}\label{ }
\nn { }&&\Big ( \pi_{(\beta,l)} \big (f(\alpha,\beta, \cdot) \big ) \xi(\beta) \Big )(\tilde g) \\
\nn { }&&  =~\Big ( \pi_{(\beta,l)} \big
(f(\alpha,\beta, \cdot) \big )\xi(\beta) \Big )(\exp{sX} \cdot \tilde g_1) \\
\nn && =~ \Big
(\frac{1}{2\pi}\Big)^{n-j_1+1} \cdot \frac{1}{\vert c(\beta, l)\vert} \int_{\R} \int_{G_1/P(\beta,l)} F_1(\alpha,
s-t, t, ((\beta, 0, l\vert_{\cc_1}), l_1), \tilde g_1, g_1)\\
\nn &&     \quad      \tilde \xi(\beta,0,l\vert_{\cc_1})(t)(g_1)d\dot g_1 dt\\
 \nn &&   =~
\int_{\R} \int_{G_1/P(\beta,l)} F(\alpha,(\beta,l), \exp{sX} \cdot \tilde g_1,
\exp{tX} \cdot g_1)\xi(\beta)(\exp{tX}\cdot g_1) dg_1 dt \\
\nn &&   =~ \int_{G/P(\beta,l)} F(\alpha,(\beta,l), \tilde g, g)\xi(\beta)(g) dg.
\end{eqnarray}
Hence for every $(\be,l)\in M $, we have the required result.

The algorithm used to build the retract function $f$ respects the
semi-norms defining the topology of our function spaces. So the
retract map $F\mapsto f$ is continuous.

\section{$\bf G$-prime ideals in $L^1(G) $}\label{primeideals}

In this section, we will study the structure of the $A $-prime ideals in $L^1(G) $ by using the retract theorem.

\subsection{ A retract defined on closed orbits}
Let $\bf G$ be a  Lie group  of automorphisms of a connected, simply
connected, nilpotent Lie group $G=\exp{\g}$ containing the inner automorphisms of $G $. For instance take any simply connected Lie group $\bf G $ and let $G $ be the nilradical of $\bf G $.

Let $l_0\in \g^*$ be fixed, we consider the orbit $\Om=\Omega_{l_0}:= {\bf G} \cdot l_0$ in $\g^* $, let $O=O_{l_0} $ be the $G $-orbit of $l_0 $. We assume that $\Om $ is locally closed in $\g^*$. In particular we can write
\begin{eqnarray*}
 \Om=\ol{\Om}\cap U,
\end{eqnarray*}
where $\ol{\Om} $ denotes the closure of $\Om $ in $\g^* $ and $U $ is  $\bf G $-invariant open subset of $\g^* $. It is then a
 $\bf G $-invariant smooth  sub-manifold of $\g^* $ diffeomorphic to the manifold $\textbf{G}/ {\textbf{G}}_{l_0} $, where $\textbf{G}_{l_0} $ denotes the stabiliser  $\textbf{G}_{l_0}:=\{\al\in \textbf{G}; \al\cdot l_0=l_0\}.
 $ The $\bf G $-orbit ${\bf G}\cdot (G\cdot l_0) $ in the orbit space $\g^*/G $ is then locally closed and homeomorphic to the quotient ${\bf G}/{\bf G}_{O} $, where ${\bf G}_O $ is the stabiliser of the set $O $ in $\bf G $. In fact, we have that ${\bf G}_O=G \cdot{\bf G}_{l_0}$.

For a Jordan-H\"older basis $\ZZ=\{Z_1,\cdots, Z_n\} $ of $\g $ and  ${\bf g}\in {\bf G}$, let
\begin{eqnarray*}
 {\bf g\cdot \ZZ}:=\{\Ad({\bf g})Z_1,\cdots, \Ad({\bf g})Z_n\},
 \end{eqnarray*}
 which is again  a Jordan-H\"older basis of $\g  $.
For every index set $I$, we have the following relation (see \cite{Lu.Mo.Sc.}):
\begin{eqnarray}\label{shiI}
 \Ad^*({\bf g})\g^*_{I,{\bf g}\cdot \ZZ}=\g^*_{I,\ZZ},\quad {\bf g}\in {\bf G}.
 \end{eqnarray}
For an index set $I $ and a Jordan-H\"older basis $\ZZ $ of $\g $, recall that
\begin{eqnarray*}
 \s_I:=\sum_{i\in I} \R Z_{i}^*, \quad  \Si_{I,\ZZ}:=\s_I\cap \g^*_{I,\ZZ},
 \end{eqnarray*}
and the mapping $ E_{I}:\R^d\ti \Si_{I,\ZZ}\to \g^*_{I,\ZZ}$ is given by
\begin{eqnarray*}
E_I(s_1,t_1,\cdots, s_r,t_r; l):= \Ad^*(\exp{s_1Z_{j_1}}\exp{t_1 Z_{k_1}}\cdots \exp{s_r Z_{j_r}}\exp{t_r Z_{k_r}}) l.
 \end{eqnarray*}
We have that $E_I$ is a bijection and $E_I(\R^d\ti \{l\}) $ is the $G $-orbit of $l $.
Let
\begin{eqnarray*}
 \Upsilon: \g^*_{I,\ZZ}\to \Si_{I,\ZZ}; \quad \Upsilon(l):=\Ad^*({\bf G})l\cap \Si_{I,\ZZ}= p_{\Si_{I,\ZZ}}(E_I\inv(l)),
 \end{eqnarray*}
where $p_{\Si_{I,\ZZ}} $ is the projection of $\R^d\ti \Si_{I,\ZZ}$ onto $\Si_{I,\ZZ} $.

For the orbit $\Omega $, we need to construct a finite partition of unity $(\ps_i)_{i\in \GA}  $ consisting of smooth $G $-invariant functions  $\ps_i: \Om \to \R_+$ such that for every $i\in \GA $ the support of each function $\ps_i $ is contained in an open subset of $\g^*_{I,{\bf g_i}\cdot \ZZ} $ for some ${\bf g}_i\in {\bf G}$.
In order to do that let $\va: \R \to \R_+$ be a smooth function with compact support which  vanishes in a neighbourhood of 0. We define a function $\ps: \g^*_{\leq I}\to \R_+ $ by
\begin{eqnarray*}
 \ps(l):= \va(P_I(\Upsilon(l)) \, \text{ if } l\in\g^*_I \quad \text{and} \quad \ps(l):= 0 \, \text{ if } l\in \g^*_{<I},
 \end{eqnarray*}
where $P_I $ is a smooth function on $\B \ti \g^*$ defined in Section \ref{step}. We see that $\ps $ is smooth (since $\va $ vanishes in a neighbourhood of 0) and is $G $-invariant by the construction.
Let
\begin{eqnarray*}
 U_{I, \ZZ}:= \{l\in \Om; \ps(l)\ne 0\}.
\end{eqnarray*}
Now assume that
 $\g^*_I= \g^*_{I,\ZZ}$ be the maximal layer with respect to $\ZZ $ such that
$\Om\cap \g^*_{I,\ZZ}\neq \emptyset$. We have that
$\Om\cap \g^*_{I,{\bf g}\cdot \ZZ}\neq \emptyset $ and $\Om\cap \g^*_{I',{\bf g}\cdot \ZZ}= \emptyset $ for ${\bf g}\in {\bf G}$ and $I'> I $. Moreover, $U_{I,\ZZ}$ is a non-empty open subset of $\Om $ contained in $\g^*_{I} $ and
\begin{eqnarray*}
 \Om\subset \bigcup_{{\bf g\in G}}\Ad^*({\bf g})U_{I, \ZZ}.
\end{eqnarray*}
Let $C $ be  a compact subset of $\g^* $ contained in $\Om $, then there exists a finite subset $\GA\subset \bf G $ such that
\begin{eqnarray*}
 C\subset \bigcup_{{\bf g}\in \GA}\Ad^*({\bf g})U_{I,\ZZ}.
 \end{eqnarray*}
Hence there is  a  finite partition of unity $(\ps_i)_{i\in \GA}  $ consisting of smooth  $G $-invariant functions  $\ps_i: \Om \to \R_+$ such that the support of each function $\ps_i $ is contained in $\Ad^*({\bf g}_i)U_{I,\ZZ}\subset \g^*_{I,{\bf g}_i\cdot \ZZ} $ for every ${\bf g}_i\in \GA $.

Suppose we have a smooth adapted operator field $F $ on $\Om $ supported on $G\cdot C $, we can write
\begin{eqnarray*}
 F=\sum_{i\in \GA} \ps_i F.
 \end{eqnarray*}
According to the retract theorem, for each $i\in \GA$ there is a (retract) Schwartz function $f_i $ on $G $ such that
\begin{eqnarray*}
 \pi_{l}(f_i)=op_{\ps_iF(l)}
 \end{eqnarray*}
for every $l\in\Om $. For $f:=\sum_{i\in \GA} f_i $, we have that
\begin{eqnarray*}
 \pi_{l}(f)&=&\sum_{i\in \GA} \pi_{l}(f_i)\\
\nn  &= & \sum_{i\in \GA}\psi_i op_{F(l)}\\
\nn  &= & op_{F(l)}.
 \end{eqnarray*}
Hence for every smooth adapted kernel function supported on $G\cdot C $, we have build a retract function.

\subsection{$\bf G$-prime ideals}\label{prime1}

Let us first recall the definition of $\bf G$-prime ideals.

\begin{definition}
{\rm A two-sided  closed ideal $I$ in $\l1G $ is called  $\bf G$-prime, if $I$ is
$\bf G$-invariant and if, for all $\bf G$-invariant two-sided ideals $I_1$ and $I_2$
of $L^1(G)$, the following implication holds
$$I_1 * I_2 \subset I \Rightarrow I_1 \subset I \text{ or } I_2 \subset I.$$}
\end{definition}

Denote by $Prim^*(G) $ the collection of all the kernels of irreducible unitary representations of $L^1(G) $.
For  a closed subset $C $ of $Prim^*(G) $, let
\begin{eqnarray*}
 \ker C:= \bigcap_{P\in C}P.
 \end{eqnarray*}
For a subset $I $ of $\l1G $, denote by $h(I) $ the subset
\begin{eqnarray*}
 h(I):= \{P\in Prim^*(G); I \subset P\}.
 \end{eqnarray*}
The set $h(I) $ is then closed in $Prim^*(G) $ with respect to the Fell  topology.

We have the following result for $\bf G$-prime ideals of $L^1(G)$ which can be viewed as an application of the retract theorem.

\begin{theorem}\label{prime}

Let $G $ be a simply connected, connected nilpotent Lie group and let $\bf G $ be a Lie group of automorphisms of $G $ containing the inner automorphisms, which acts smoothly on the group $G $, such that every $\bf G $-orbit in $\g^* $ is locally closed.
If $I$ is a proper $\bf G$-prime ideal of $L^1(G)$, then
there exists an $\bf G$-orbit $\Omega_{l_0}$ in $\g^*$ such that
$$I= \ker{\Omega_{l_0}}.$$
Moreover, all the kernels of $\bf G$-orbits are $\bf G$-prime ideals.
\end{theorem}

\begin{proof}
For any ${\bf G}$-orbit $\Om $ in $\g^*$, the retract theorem tells us that the Schwartz functions contained in $\ker {\Om} $ are dense in $\ker \Om $ (see \cite[proof of Proposition 4.1]{Lu.Mo.1} and \cite{La.Mo.1}). From the proof of \cite[Theorem 1.2.12]{Lu-Mo.2},
it follows that the hull of a prime ideal $I $ is the closure of an ${\bf G}$-orbit in $Prim^*(G)\simeq \widehat G $. On the other hand,
the density of $\S(G)\cap \ker \Om $ implies that $\ker \Om ^N$ is contained in the minimal ideal $J(\Om) $ with hull $\Om $ for some $N\in \N$. This tells us that $\ker\Om^N\subset J(\Om)\subset I $, since the minimal ideal  with hull $\Om $ is contained in every ideal with hull $\Om $. Since $I $ is ${\bf G}$-prime, we have that $I=\ker\Om $.

Obviously the ideal $\ker \Om $ is ${\bf G}$-prime for any $\textbf{G} $-orbit $\Om $ in $\g^* $. To see this, let $I_1$ and $I_2 $ be two ${\bf G}$-invariant ideals of $\l1G $ such that $I_1\ast I_2\subset \ker \Om $. This means that $I_1\ast I_2\subset \ker\Om \subset \ker{\pi_l} $ for some $l \in \Om $. We have then either $I_1 $ or $I_2 $ is contained in $\ker{\pi_l} $, since $\pi_l $ is irreducible. But if $I_1 $ is contained in $\ker {\pi_l} $, it is also contained in $\ker{ \pi_{k\cdot l}} $ since $I_1 $ is $\bf G $-invariant. Hence $I_1\subset \ker\Om $ and the proof is complete.
\end{proof}

\bigskip

\noi Ying-Fen Lin, {\it Pure Mathematics Research Centre,
Queen's University Belfast,
 Belfast, BT7 1NN, U.K.
E-mail: y.lin@qub.ac.uk}

\medskip
 \noi Jean Ludwig, {\it  Universit\'e de Lorraine,
Institut Elie Cartan de Lorraine,
UMR 7502, Metz, F-57045, France.
E-mail: jean.ludwig@univ-lorraine.fr}

\medskip
\noi Carine Molitor-Braun, {\it  Unit\'e de Recherche en
Math\'ematiques, Universit\'e du Luxembourg, 6, rue
Coudenhove-Kalergi, L-1359 Luxembourg, Luxembourg.
E-mail: carine.molitor@uni.lu}

\end{document}